\documentclass[12pt,oneside,reqno]{amsart}

\usepackage{amsmath,amssymb,amsthm,textcomp}
\usepackage{amsfonts,graphicx}
\usepackage[mathscr]{eucal}
\usepackage{float}
\usepackage{color}
\usepackage{csquotes}
\usepackage[backend=bibtex,%
firstinits=true,%
doi=false,%
isbn=true,%
url=false,%
maxnames=99]{biblatex}%

\AtEveryBibitem{\clearfield{issn}}
\AtEveryCitekey{\clearfield{issn}}
\numberwithin{equation}{section}
\DeclareNameAlias{sortname}{last-first}
\theoremstyle{definition}
\usepackage{mathtools}
\numberwithin{equation}{section}

\newcommand{\ncom}{\newcommand}

\ncom{\beq}{\begin{equation}}
\ncom{\eeq}{\end{equation}}
\ncom{\bea}{\begin{eqnarray*}}
\ncom{\eea}{\end{eqnarray*}}
\ncom{\beqa}{\begin{eqnarray}}
\ncom{\eeqa}{\end{eqnarray}}
\ncom{\nno}{\nonumber}
\ncom{\non}{\nonumber}
\ncom{\ds}{\displaystyle}
\ncom{\half}{\frac{1}{2}}
\ncom{\mbx}{\makebox{.25cm}}
\ncom{\hs}{\mbox{\hspace{.25cm}}}
\ncom{\rar}{\rightarrow}
\ncom{\Rar}{\Rightarrow}
\ncom{\noin}{\noindent}
\ncom{\bc}{\begin{center}}
\ncom{\ec}{\end{center}}
\ncom{\sz}{\scriptsize}
\ncom{\rf}{\ref}
\ncom{\s}{\sqrt{2}}
\ncom{\sgm}{\sigma}
\ncom{\Sgm}{\Sigma}
\ncom{\psgm}{\sigma^{\prime}}
\ncom{\dt}{\delta}
\ncom{\Dt}{\Delta}
\ncom{\lmd}{\lambda}
\ncom{\Lmd}{\Lambda}
\ncom{\Th}{\Theta}
\ncom{\e}{\eta}
\ncom{\eps}{\epsilon}
\ncom{\pcc}{\stackrel{P}{>}}
\ncom{\lp}{\stackrel{L_{p}}{>}}
\ncom{\dist}{{\rm\,dist}}
\ncom{\sspan}{{\rm\,span}}
\ncom{\re}{{\rm Re\,}}
\ncom{\im}{{\rm Im\,}}
\ncom{\sgn}{{\rm sgn\,}}
\ncom{\ba}{\begin{array}}
\ncom{\ea}{\end{array}}
\ncom{\hone}{\mbox{\hspace{1em}}}
\ncom{\htwo}{\mbox{\hspace{2em}}}
\ncom{\hthree}{\mbox{\hspace{3em}}}
\ncom{\hfour}{\mbox{\hspace{4em}}}
\ncom{\vone}{\vskip 2ex}
\ncom{\vtwo}{\vskip 4ex}
\ncom{\vonee}{\vskip 1.5ex}
\ncom{\vthree}{\vskip 6ex}
\ncom{\vfour}{\vspace*{8ex}}
\ncom{\norm}{\|\;\;\|}
\ncom{\integ}[4]{\int_{#1}^{#2}\,{#3}\,d{#4}}
\ncom{\vspan}[1]{{{\rm\,span}\{ #1 \}}}
\ncom{\dm}[1]{ {\displaystyle{#1} } }
\ncom{\ri}[1]{{#1} \index{#1}}

\newtheorem{theorem}{\bf Theorem}[section]
\newtheorem{remark}{\bf Remark}[section]
\newtheorem{proposition}{Proposition}[section]

\newtheoremstyle
    {remarkstyle}
    {}
    {11pt}
    {}
    {}
    {\bfseries}
    {:}
    {     }
    {\thmname{#1} \thmnumber{#2} }

\theoremstyle{remarkstyle}

\def\k{\kappa}

\def\eps{\varepsilon}

\begin{document}
\title{Multivariate generalized counting process via gamma subordination}

\author[Manisha Dhillon]{Manisha Dhillon}
\address{Manisha Dhillon, Department of Mathematics, Indian Institute of Technology Bhilai, Durg, 491002, India.}
\email{manishadh@iitbhilai.ac.in}
\author[Kuldeep Kumar Kataria]{Kuldeep Kumar Kataria}
\address{Kuldeep Kumar Kataria, Department of Mathematics, Indian Institute of Technology Bhilai, Durg, 491002, India.}
\email{kuldeepk@iitbhilai.ac.in}
\author[Shyan Ghosh]{Shyan Ghosh}
\address{Shyan Ghosh, Department of Mathematics, Indian Institute of Technology Bhilai, Durg, 491002, India.}
\email{shyanghosh@iitbhilai.ac.in}

\date{\today}

\begin{abstract}
In this paper, we study a multivariate gamma subordinator whose components are independent gamma processes subject to a random time governed by an independent negative binomial process. We derive the explicit expressions for its joint Laplace–Stieltjes transform, its probability density function and the associated governing differential equations. Also, we study a time-changed variant of the multivariate generalized counting process where the time is changed by an independent multivariate gamma subordinator. For this time-changed process, we obtain the corresponding Lévy measure and probability mass function. Later, we discuss an application of the time-changed multivariate generalized counting process to a shock model.
\end{abstract}

\maketitle

\section{Introduction}
Di Crescenzo {\it{et al.}} (2016) intoduced and studied a generalization of the Poisson process by considering the possibility of finitely many arrivals in an infinitesimal time interval. It is known as the generalized counting process (GCP), and we denote it by $\{M(t)\}_{t\ge 0}$. It is a L\'evy process which performs independent jumps of size $1,2,\dots,k$ with positive rates $\lambda_1,\lambda_2,\dots,\lambda_k$, respectively. Its state probabilities are given by (see Di Crescenzo {\it{et al.}} (2016))
\begin{equation*}
p(n,t)=\mathrm{Pr}\{M(t)=n\}=\sum_{\Omega(k,n)}\prod_{j=1}^{k} \frac{(\lambda_{j}t)^{n_j}}{n_j!}e^{-\lambda_j t}, \ n\ge 0,
\end{equation*}
where $\Omega(k,n)=\big\{(n_1,n_2,\dots,n_k): \sum_{j=1}^{k}jn_j=n,\, n_j\in\mathbb{N}_0\big\}$.
Also, its L\'evy measure and  probability generating function (pgf) are given by (see Kataria and Khandakar (2022a))
\begin{equation*}
\Pi_{M(t)}(\mathrm{d}x)=\sum_{j=1}^{k}\lambda_j \delta_j(\mathrm{d}x),
\end{equation*}
and
\begin{equation}\label{pgf_gcp}
G_{M(t)}(u)=\mathbb{E}(u^{M(t)})=\exp\Big(-t\sum_{j=1}^{k}\lambda_j (1-u^j)\Big),\ |u|\le 1,
\end{equation}
respectively, where $\delta_j$'s denote the Dirac measures. Also, Di Crescenzo {\it{et al.}} (2016)  studied a time-changed variant of the GCP where an independent inverse stable subordinator is utilized for the random time change. Such time-changed processes are of broad interest in statistical physics as they are closely related to anomalous diffusion phenomena (see Beghin \textit{et al.} (2020)). As the time-changed processes exhibit long-range dependence property, capturing the long-term memory effects, they have several applications in areas, such as finance, risk theory, and internet data traffic modelling. Two mainly studied time-changed variants of the GCP are space fractional and time fractional GCPs which are obtained by considering an independent stable subordinator and its inverse as the time changing components, respectively. For more details on GCP and its time-changed variants, we refer the reader to Kataria and Khandakar (2022a), Buchak and Sakhno (2024) and Kataria {\it et al.} (2025). Moreover,
for martingale characterizations of the GCP and its time-changed variants, we refer the reader to Dhillon and Kataria (2024).

In recent past, the subordination of L\'evy processes has been extensively studied (see Bertoin (1996), (1999), and Sato (1999)). In Sato (1999, Chapter 6), the classical framework of multivariate subordination is discussed. Beghin and Macci (2016) studied a multivariate version of the time-changed Poisson process. An application of multivariate subordination to financial models is given by Semeraro (2022). Cha and Giorgio (2018) introduced a class of bivariate counting processes with marginal regularity property that have applications to certain shock models. Moreover, an application of the bivariate  Poisson process time-changed by an independent stable subordinator to the competing risks model is given by Di Crescenzo and Meoli (2022). Barndorff-Nielsen {\it et al.} (2001) considered $\bar{T}(t)=(T_1(t), T_2(t),\dots,T_d(t))$ as a $d$-dimensional subordinator, and introduced a multivariate subordinated process on $\mathbb{R}^d$ as follows: 
\begin{equation*}
 \bar{Y}(t)=(X_1(T_1(t)),X_2(T_2(t)),\dots,X_d(T_d(t))),
\end{equation*}
where the multivariate processes $\bar{X}(t)=(X_1(t),X_2(t),\dots,X_d(t))$ and $\bar{T}(t)=(T_1(t),T_2(t)$, $\dots,T_d(t))$, $t\ge0$ are assumed to be independent. Beghin \textit{et al.} (2020) studied multivariate inverse subordinators where the component processes are right-continuous hitting time of the component processes $\{T_i(t)\}_{t\ge0}$. Recently, Meoli (2025) studied a bivariate Poisson process subordinated by an independent bivariate gamma process with conditional independent component processes, where an application of the obtained results to a shock model is discussed. 

The outline of this paper is as follows:

First, we study a multivariate subordinator and then use it as a time changing component to study a time-changed variant of the multivariate GCP. In Section 2, we set some notations and recall some preliminary results which will be used throughout the paper. In Section 3, we study a multivariate gamma subordinator in which the component processes are conditionally independent given a negative binomial process. For this multivariate subordinator, we give the explicit expressions for the joint Laplace-Stieltjes transform of its component processes, and obtain its probability density function (pdf) and derive the governing differential equations of its pdf. Also, we obtain the covariance and codifference of its component processes. In Section 4, we study a time-changed variant of the multivariate GCP by time changing it by an independent multivariate subordinator. Several distributional properties of the resulting time-changed multivariate GCP are derived that includes the explicit expressions for its state probabilities, pgf, their governing system of differential equations as well as the covariance and codifference of the component processes. Later, we discuss an application of this time-changed variant to a shock model. For some particular cases, we obtain the survival function of system's failure time.
\section{Preliminaries}
In this section, we first introduce the notation that will be used throughout the paper. We then recall several definitions and established results related to special functions, the gamma process, and the multivariate GCP, \textit{etc.} These preliminaries will be used in the subsequent sections.

Throughout the paper, we will use the following notations: Let $\mathbb{N}=\{1,2,3,\dots\}$ denote the set of positive integers, $\mathbb{N}_{0}=\{0,1,2,\dots\}$ be the set of non-negative integers, $\mathbb{N}_{0}^q=\mathbb{N}_{0}\times\mathbb{N}_{0}\times\dots\times\mathbb{N}_{0}$ ($q$ copies), $\mathbb{I}_{A}$ be the indicator function on set $A$,   $\bar{n}=(n_1,n_2,\dots,n_q)$, $\bar{0}=(0,0,\dots,0)$ be $q$-tuple zero vector and $\omega=\sqrt{-1}$. By $\bar{n}\ge \bar{m}$ ($\bar{n}> \bar{m}$), we mean $n_i\ge m_i$ ($n_i> m_i$) for all $i\in \{1,2,\dots,q\}$. Also, $\bar{n}\succ \bar{m}$ denotes that $n_i\ge m_i$ for all $i\in\{1,2,\dots,q\}$ and $\bar{n}\neq \bar{m}$. Also, for $i\ge1$,  $\Omega(k_i,n_i)=\big\{(n_{i1}, n_{i2}, \dots, n_{ik_i}):\sum_{j_i=1}^{k_i}j_in_{ij_i}=n_i,\, n_{ij_i}\in \mathbb{N}_{0}\big\}$.

\subsection{Generalized hypergeometric function}
It is defined as follows (see Kilbas \textit{et al.} (2006), Eq. (1.6.28)):
\begin{equation*}
	\begin{aligned}\label{GenHyp}
		_pF_q(a_1,a_2,\dots,a_p; b_1,b_2,\dots,b_q;z)\coloneqq\sum_{k=0}^{\infty}\frac{(a_1)_k(a_2)_k\dots(a_p)_k}{(b_1)_k(b_2)_k\dots(b_q)_k}\frac{z^k}{k!},\, z\in\mathbb{C},\ |z|<1,
	\end{aligned}
\end{equation*}
where $(x)_k$ denotes the Pochhammer symbol, that is,
\begin{equation}\label{pochha}
		(x)_k =\begin{cases*}
			x(x+1)\dots(x+k-1), \  k\ge1,\\
			1, \  k=0. 
		\end{cases*}
\end{equation}

\subsection{Gamma process}
The pdf of gamma process $\{Z(t)\}_{t\ge 0}$ is given by
\begin{equation}\label{gamma}
	f(x,t)=\frac{a^{\lambda t}}{\Gamma(\lambda t)}x^{\lambda t-1}e^{-ax},\ x>0. 
\end{equation}
That is, its marginals $Z(t)$ are gamma distributed with shape parameter $\lambda t$, $\lambda>0$ and rate parameter $a > 0$. We write $Z(t)\sim \Gamma(\lambda t,a)$.
Thus, its mean and variance are given by
\begin{equation}\label{mean_gamma}
\mathbb{E}(Z(t))=\frac{\lambda t}{a}\ \ \text{and}\ \ \operatorname{Var}(Z(t))=\frac{\lambda t}{a^2},
\end{equation}
respectively. 
Its Laplace transform is given by (see Applebaum (2009), p. 55)
\begin{equation}\label{LT_gamma}
	\mathbb{E}(e^{-uZ(t)})=\Big(1+\frac{u}{a}\Big)^{-\lambda t},\ u\ge 0.
\end{equation}

\subsection{Negative binomial process}
The probability mass function (pmf) of negative binomial process $\{B^-(t)\}_{t\ge 0}$ with parameter $0<\theta<1$ is given by (see Barndorff-Nielsen (2001), Example 2.1)
\begin{equation}\label{negbin}
	\mathrm{Pr}\{B^-(t)=n\}= (1-\theta)^t{n+t-1 \choose n}\theta^n,\ n\in\mathbb{N}_0.
\end{equation}
Thus, its mean and variance are
\begin{equation}\label{meanvarnegbin}
\mathbb{E}(B^-(t))=\frac{\theta t}{1-\theta}\ \ \text{and}\ \ \operatorname{Var}(B^-(t))=\frac{\theta t}{(1-\theta)^2},
\end{equation}
respectively. 
Also, its pgf and Laplace transform are given by
\begin{equation}\label{pgf_negbin}
	G_{B^-(t)}(u)=\mathbb{E}(u^{B^-(t)})=\Big(\frac{1-\theta}{1-\theta u}\Big)^t,\ |u|<\frac{1}{1-\theta}
\end{equation}
and
\begin{equation*}\label{LT_negbin}
	\mathbb{E}(e^{-uB^-(t)})=\Big(\frac{1-\theta}{1-\theta e^{-u}}\Big)^t,\ u\ge 0,
\end{equation*}
respectively.
\subsection{L\'evy-Khintchine formula}
The characteristic function of a L\'evy process $\{X(t)\}_{t\ge 0}$ is given by the following L\'evy-Khintchine formula (see Applebaum (2009), p. 45):
\begin{equation}\label{Lev-Khint}
	\mathbb{E}\big(e^{\omega u X(t)}\big)=\exp\Big(t\Big(\omega b u-\frac{1}{2}\sigma^2 u^2+\int_{\mathbb{R}\setminus\{0\}}(e^{\omega ux}-1-\omega ux \mathbb{I}_{\{|x|<1\}})\,\Pi_{X(t)}(\mathrm{d}x)\Big)\Big),
\end{equation}
where $b\in\mathbb{R}$, $\sigma\in\mathbb{R}$ and $\Pi_{X(t)}(\cdot)$ is the L\'evy measure of $\{X(t)\}_{t\ge 0}$ concentrated on $\mathbb{R}\setminus\{0\}$ which satisfies $\int_{\mathbb{R}\setminus\{0\}}\min\{1,x^2\}\,\Pi_{X(t)}(\mathrm{d}x)<\infty$.
\subsection{Codifference of random variables} The codifference of two random variables $X$ and $Y$ is defined as follows (see Kokoszka and Taqqu (1996), Eq. (1.7)):
\begin{equation}\label{codiff}
\tau(X,Y)\coloneq\ln\mathbb{E}(e^{\omega(X-Y)})-\ln\mathbb{E}(e^{\omega X})-\ln\mathbb{E}(e^{-\omega Y}).
\end{equation}

\subsection{Multivariate GCP}
For $i=1,2,\dots,q$, let $\{M_i(t)\}_{t\ge0}$'s be $q$ independent GCPs which perform independently $k_i$ kinds of jumps of size $j_i$ with positive rate $\lambda_{ij_i}$, $j_i=1, 2, \dots, k_i$. Then, a multivariate GCP $\{\bar{M}(t)\}_{t\ge0}$ is defined as (see Kataria and Dhillon (2025), Section 3)
\begin{equation*}\label{mvgcp}
\bar{M}(t)\coloneqq (M_1(t), M_2(t),\dots,M_q(t)),\, t\ge0. 
\end{equation*}

For $\bar{n}\geq \bar{0}$, its state probabilities $p(\bar{n},t)=\mathrm{Pr}\{\bar{M}(t)=\bar{n}\}$ are given by (see Kataria and Dhillon (2025), Eq. (3.7))
\begin{equation}\label{jopmfgcp}
p(\bar{n},t)=\prod_{i=1}^{q}\sum_{\Omega(k_i,n_i)}\prod_{j_i=1}^{k_i}\frac{(\lambda_{ij_i}t)^{n_{ij_i}}}{n_{ij_i}!}e^{-\lambda_{ij_i} t}.
\end{equation}

\section{Multivariate Gamma Subordinator}
Let $\{Z_1(t)\}_{t\ge0}$, $\{Z_2(t)\}_{t\ge0}$, $\dots$,    $\{Z_q(t)\}_{t\ge0}$ be independent gamma L\'evy processes with marginal distribution $\Gamma(\lambda t, a_1)$, $\Gamma(\lambda t, a_2)$, $\dots$, $\Gamma(\lambda t, a_q)$, respectively. Here, $\lambda t>0$ is the shape parameter and $a_i>0$ are the rate parameters. Also, let $\{B^-(\lambda t)\}_{t\ge 0}$ be a negative binomial L\'evy process with parameter $0<\theta<1$. Then, for $i=1,2,\dots,q$, we consider the following time-changed processes: 
\begin{equation}\label{multi_gamma}
G_i(t)= Z_i\bigl(t + \lambda^{-1} B^{-}(\lambda t)\bigr), \ t\ge0,
\end{equation}
where $\{B^{-}(\lambda t)\}_{t\ge 0}$ is independent of $\{Z_i(t)\}_{t\ge0}$.

Note that $\{G_i(t)\}_{t\ge0}$ is a gamma process with distribution $\Gamma(\lambda t, (1 - \theta) a_i)$ (see Barndorff-Nielsen \textit{et al.} (2001), Example 2.2). We define a multivariate gamma subordinator as follows:
\begin{equation}\label{gamma_sub}
\bar{G}(t) \coloneqq (G_1(t), G_2(t),\dots,G_q(t)),
\end{equation}
 where the component processes $\{G_i(t)\}_{t\ge 0}$, $i=1,2,\dots,q$ are conditionally independent given $\{B^{-}(\lambda t)\}_{t\ge 0}$.
\begin{proposition}
The pdf of $\{\bar{G}(t)\}_{t\ge 0}$ is given by 
\small{\begin{equation}\label{gamma_density}
g(\bar{x},t)= \frac{(1-\theta)^{\lambda t}}{\Gamma(\lambda t)}\sum_{n=0}^{\infty}\frac{\theta^n}{n!(\Gamma(n+\lambda t))^{q-1}}\prod_{i=1}^{q}\frac{(a_ix_i)^{n+\lambda t}}{x_i}e^{-a_ix_i},\ \bar{x}>\bar{0}, \ t\ge0,
	\end{equation}}
where $a_i> 0$ and $0<\theta<1$.
\end{proposition}
\begin{proof}
Let $f_i(x_i,t)$ be the marginal pdf of gamma process $\{Z_i(t)\}_{t\ge 0}$. Then, from \eqref{gamma_sub}, we have
\begin{align*}
		g(\bar{x},t)&=\sum_{n=0}^{\infty}\mathrm{Pr}\{B^-(\lambda t)=n\}\prod_{i=1}^{q}f_i(x_i,t+\lambda^{-1}n)\\
		&=\sum_{n=0}^{\infty}(1-\theta)^{\lambda t}{n+\lambda t-1 \choose n}\theta^n\prod_{i=1}^{q}\frac{a_i^{n+\lambda t}}{\Gamma(n+\lambda t)}x_i^{n+\lambda t-1}e^{-a_ix_i},\ \text{(using \eqref{gamma} and \eqref{negbin})}\\
		&=(1-\theta)^{\lambda t}\Big(\prod_{i=1}^{q}\frac{e^{-a_ix_i}(a_ix_i)^{\lambda t}}{x_i}\Big)\sum_{n=0}^{\infty}{n+\lambda t-1 \choose n}\frac{\theta^n}{(\Gamma(n+\lambda t))^q}\prod_{i=1}^{q}(a_ix_i)^{n}.
	\end{align*}
 This completes the proof.
\end{proof}

\begin{remark}
On substituting $q=2$ and $a_1=a_2=a$ in \eqref{gamma_density}, we get the pdf of bivariate gamma subordinator (see Barndorff-Nielsen {\it et al.} (2001), Example 2.2). Also, for $q=1$, the pdf in \eqref{gamma_density} reduces to that of gamma subordinator (see Applebaum (2009), p. 54).
\end{remark}

\begin{proposition}
Let $\mathbb{I}_{A}$ be the indicator function on set $A$. Then, for any $i=1$, $2$, $\dots$, $q$ and $j=1$, $2$, $\dots$, $q$, the covariance of $\{G_i(t)\}_{t\ge 0}$ and $\{G_j(t)\}_{t\ge 0}$ is given by
\begin{equation*}
\operatorname{Cov}(G_i(t),G_j(t))=\frac{\lambda t}{a_ia_j(1-\theta)^2}(\theta \mathbb{I}_{\{i\ne j\}}+\mathbb{I}_{\{i=j\}}).
\end{equation*}
\end{proposition}
\begin{proof}
For $i=j$, by using Eq. (3) of Leonenko et al. (2014), we have
\begin{align}\label{cov1}
\operatorname{Cov}(G_i(t),G_j(t))&=(\mathbb{E}(Z(1)))^2\operatorname{Var}(t+\lambda^{-1}B^-(\lambda t))+\operatorname{Var}(Z(1))\mathbb{E}(t+\lambda^{-1}B^-(\lambda t))\nonumber\\
&=\frac{\lambda t}{(a_i(1-\theta))^2},
\end{align}
where we have used \eqref{mean_gamma} and \eqref{meanvarnegbin}.

For $i\ne j$, by using the independence of $\{Z_i(t)\}_{t\ge0}$'s, we get
 \begin{align}\label{jomean}
\mathbb{E}(G_i(t)G_j(t))&=\sum_{n=0}^{\infty}\mathrm{Pr}\{B^-(\lambda t)=n\} \mathbb{E}(Z_i(t+\lambda^{-1}n))\mathbb{E}(Z_j(t+\lambda^{-1}n))\nonumber\\
&=\sum_{n=0}^{\infty}(1-\theta)^{\lambda t}{n+\lambda t-1 \choose n}\theta^n \frac{(n+\lambda t)^2}{a_ia_j}\nonumber\\
&=\frac{\lambda t (\lambda t+\theta)}{a_ia_j(1-\theta)^2}.
\end{align}
Also, by using Eq. (2) of Leonenko et al. (2014), we obtain
\begin{equation}\label{Gi_mean}
\mathbb{E}(G_i(t))= \mathbb{E}(Z(1))\mathbb{E}(t+\lambda^{-1}B^-(\lambda t))=\frac{\lambda t}{a_i(1-\theta)},
\end{equation}
which follows on using \eqref{mean_gamma} and \eqref{meanvarnegbin}.

Thus, in this case, by using \eqref{jomean} and \eqref{Gi_mean}, the covariance of $\{G_i(t)\}_{t\ge 0}$ and $\{G_j(t)\}_{t\ge 0}$ is given by
\begin{equation}\label{cov2}
\operatorname{Cov}(G_i(t),G_j(t))=\frac{\lambda t\theta}{a_ia_j(1-\theta)^2}.
\end{equation}
Finally, the result follows from \eqref{cov1} and \eqref{cov2}.
\end{proof}
\begin{proposition}
For $i=1$, $2$, $\dots$, $q$ and $j=1$, $2$, $\dots$, $q$, the codifference of $\{G_i(t)\}_{t\ge 0}$ and $\{G_j(t)\}_{t\ge 0}$ is 
{\footnotesize\begin{align*}
\tau(G_i(t),G_j(t))&=\lambda t \ln\Big(\frac{a_ia_j(1-\theta)}{a_ia_j(1-\theta)+\omega a_i-\omega a_j+1}\Big)\mathbb{I}_{\{i\ne j\}} -\lambda t\ln\Big(\frac{a_ia_j(1-\theta)^2}{(a_i(1-\theta)-\omega)(a_j(1-\theta)+\omega)}\Big).
\end{align*}}
\end{proposition}
\begin{proof}
From \eqref{codiff}, we have
\begin{equation}\label{codiff1}
\tau(G_i(t),G_j(t))\coloneq\ln\mathbb{E}(e^{\omega(G_i(t)-G_j(t))})-\ln\mathbb{E}(e^{\omega G_i(t)})-\ln\mathbb{E}(e^{-\omega G_j(t)}).
\end{equation}
Also, from \eqref{LT_gamma}, we get
\begin{align}\label{wGi}
\mathbb{E}(e^{\omega G_i(t)})=\Big(\frac{a_i(1-\theta)}{a_i(1-\theta)-\omega}\Big)^{\lambda t}.
\end{align}
For $i=j$, on substituting \eqref{wGi} in \eqref{codiff1}, we obtain
\begin{equation}\label{tauij}
\tau(G_i(t),G_i(t))= -\lambda t\ln\Big(\frac{(a_i(1-\theta))^2}{(a_i(1-\theta))^2+1}\Big).
\end{equation}
For $i\ne j$, we have 
\begin{align*}
\mathbb{E}(e^{\omega(G_i(t)-G_j(t))})
&=\mathbb{E}\big(\mathbb{E}(e^{\omega G_i(t)}e^{-\omega G_j(t)}|B^-(\lambda t))\big)\nonumber\\
&=\sum_{n=0}^{\infty}\mathbb{E}\big(e^{\omega Z_i(t+\lambda^{-1} n)}e^{-\omega Z_j(t+\lambda^{-1} n)}\big)(1-\theta)^{\lambda t}{n+\lambda t-1\choose n}\theta^n\nonumber\\
&=\sum_{n=0}^{\infty}\Big(\frac{a_ia_j}{a_ia_j(1-\theta)+\omega a_i-\omega a_j+1}\Big)^{n+\lambda t}(1-\theta)^{\lambda t}{n+\lambda t-1\choose n}\theta^n\nonumber\\
&=\Big(\frac{a_ia_j(1-\theta)}{a_ia_j(1-\theta)+\omega a_i-\omega a_j+1}\Big)^{\lambda t}.
\end{align*}
Thus, for $i\ne j$, we obtain
{\footnotesize\begin{equation}\label{w2Gi}
\tau(G_i(t),G_j(t))=\lambda t \Big(\ln\Big(\frac{(a_i(1-\theta)-\omega)(a_j(1-\theta)+\omega)}{(1-\theta)(a_ia_j(1-\theta)+\omega a_i-\omega a_j+1)}\Big) \Big).
\end{equation}}
Finally, the required result follows from \eqref{tauij} and \eqref{w2Gi}.
\end{proof}

\begin{proposition}
The joint Laplace-Stieltjes transform of multivariate gamma subordinator $\bar{G}(t)$ is given by
\begin{equation}\label{Ltmg}
\mathbb{E}\big(e^{-\bar{s}\cdot\bar{G}(t)}\big)=\Big(\frac{1-\theta}{\prod_{i=1}^{q}(1+s_ia_i^{-1})-\theta}\Big)^{\lambda t},\ \bar{s}>\bar{0},\ t\ge0,
\end{equation}
where $a_i>0$ and $0<\theta<1$.
\end{proposition}
\begin{proof}
	By using \eqref{gamma_density}, we have
	\begin{align*}
		\mathbb{E}\big(e^{-\bar{s}\cdot\bar{G}(t)}\big)&=\mathbb{E}\big(e^{-\sum_{i=1}^{q}s_i G_i(t)}\big)\\
		&=\int_{0}^{\infty}\int_{0}^{\infty}\dots\int_{0}^{\infty}e^{-\sum_{i=1}^{q}s_i x_i}g(\bar{x},t)\,\mathrm{d}x_1\,\mathrm{d}x_2\dots\mathrm{d}x_q\\
		&=\frac{(1-\theta)^{\lambda t}}{\Gamma(\lambda t)}\sum_{n=0}^{\infty}\frac{\theta ^n}{n!(\Gamma(n+\lambda t))^{q-1}}\prod_{i=1}^{q}\int_{0}^{\infty}\frac{(a_ix_i)^{n+\lambda t}}{x_i}e^{-(a_i+s_i)x_i}\,\mathrm{d}x_i\\
		&=\frac{(1-\theta)^{\lambda t}}{\Gamma(\lambda t)}\sum_{n=0}^{\infty}\frac{\theta ^n}{n!(\Gamma(n+\lambda t))^{q-1}}\prod_{i=1}^{q}\frac{a_i^{n+\lambda t}\Gamma(n+\lambda t)}{(a_i+s_i)^{n+\lambda t}}\\
		&=(1-\theta)^{\lambda t}\sum_{n=0}^{\infty}{n+\lambda t-1 \choose n}\theta ^n\prod_{i=1}^{q}\Big(\frac{a_i+s_i}{a_i}\Big)^{-n-\lambda t}.
	\end{align*}
	This completes the proof.
\end{proof}
\begin{remark}
For $q=2$ and $a_1=a_2=a$, the Laplace-Stieltjes transform in \eqref{Ltmg} reduces to that of bivariate gamma subordinator (see Meoli (2025), Eq. (3.3)). Also, for $q=1$, it reduces to that of gamma subordinator (see Applebaum (2009), p. 55).
\end{remark}
\subsection{Governing differential equation of the pdf} Here, we obtain the system of differential equations that governs the pdf of multivariate gamma subordinator for a particular case in which $a_1=a_2=\dots=a_q=a$. 
	
Let $\mathcal{F}\{g(\bar{x},t);\alpha_1,\alpha_2,\dots,\alpha_q\}$ be the Fourier transform of $g(\bar{x},t)$. Then, for $a_1=a_2=\dots=a_q=a$, we have
\begin{equation}\label{fourier}
\mathcal{F}\{g(\bar{x},t);\alpha_1,\alpha_2,\dots,\alpha_q\}=\mathbb{E}\big[e^{\omega\sum_{j=1}^{q}\alpha_j G_j(t)}\big]=\bigg(\frac{(1-\theta)a^q}{\big(\prod_{j=1}^{q}(a-\omega\alpha_j)\big)-\theta a^q}\bigg)^{\lambda t}.
\end{equation}
In the next result, we will use the following shift operator (see Beghin (2014), Eq. (1.13)):
\begin{equation}\label{shiftope}
	e^{-k\partial_t}f(t)\coloneq \sum_{n=0}^{\infty}\frac{(-k\partial_t)^n}{n!}f(t)=f(t-k), \, k\in\mathbb{R},
\end{equation}
where $f\colon\mathbb{R}\to \mathbb{R}$ is an analytic function. 

\begin{theorem}
Let $\|\cdot\|$ denote the Euclidean norm and $e^{-\frac{1}{\lambda}\partial_t}$ be a shift operator as defined in \eqref{shiftope}. Then, for $a_1=a_2=\dots=a_q=a$, the pdf of the multivariate gamma subordinator satisfies the following differential equation:
{\small	\begin{equation}\label{de1}
		\Big(\sum_{m=1}^{q}a^{q-m}\sum_{1\le j_1<j_2<\dots<j_m\le q}\frac{\partial^m}{\partial x_{j_1}\partial x_{j_2}\dots\partial x_{j_m}}\Big)g(\bar{x},t)=-(1-\theta)a^q\Big(1-e^{-\frac{1}{\lambda}\partial_t}\Big)g(\bar{x},t),\ \bar{x}>\bar{0}, \  t\ge 0,
	\end{equation}}
with initial and boundary conditions $g(\bar{x},0)=\delta (\bar{x})$ and $\lim_{\|\bar{x}\|\to \infty}g(\bar{x},t)=0$, respectively.  
\end{theorem}

\begin{proof}
By the definition of Dirac delta function, we have
{\footnotesize\begin{equation*}
\delta(\bar{x})=\delta (x_1,x_2,\dots,x_q)=\prod_{j=1}^{q}\delta (x_j)=\prod_{j=1}^{q}\frac{1}{2\pi}\int_{\mathbb{R}}e^{-\omega\theta_jx_j}\,\mathrm{d}\theta_j
=\frac{1}{(2\pi)^q}\int_{\mathbb{R}^q}e^{-\omega\sum_{j=1}^{q}\theta_j x_j}\,\mathrm{d}x_1\,\mathrm{d}x_2\dots\mathrm{d}x_q.
\end{equation*}}
The initial condition follows from \eqref{fourier} and the definition of the Dirac delta function.
Also, it can be shown that the pdf  \eqref{gamma_density} of multivariate gamma subordinator satisfies the boundary condition.
	
For $q=1$, we have
{\footnotesize	\begin{equation}\label{q1_lhs}
		\mathcal{F}\Big\{\frac{\partial}{\partial x_1}g(x_1,t);\alpha_1\Big\}=(-\omega\alpha_1)\Big(\frac{(1-\theta)a}{a-\omega\alpha_1-\theta a}\Big)^{\lambda t}
	\end{equation}}
and
{\footnotesize \begin{equation*}
		\mathcal{F}\Big\{-(1-\theta)a\Big(1-e^{-\frac{1}{\lambda}\partial_t}\Big)g(x_1,t);\alpha_1\Big\}=-(1-\theta)a\Big(\frac{(1-\theta)a}{a-\omega\alpha_1-\theta a}\Big)^{\lambda t}\Big(1-\Big(\frac{(1-\theta)a}{a-\omega\alpha_1-\theta a}\Big)^{-1}\Big)
	\end{equation*}}
	which coincides with \eqref{q1_lhs}.

For $q=2$, we have
{	\begin{align}\label{q2_lhs}
		\mathcal{F}\Bigl\{\Big(&\frac{\partial^2}{\partial x_1\partial x_2}+a\Big(\frac{\partial}{\partial x_1}+\frac{\partial}{\partial x_2}\Big)\Big)g(\bar{x},t);\alpha_1,\alpha_2\Bigr\}\nonumber\\
		&=\Big((-\omega\alpha_1)(-\omega\alpha_2)+a(-\omega\alpha_1-\omega\alpha_2)\Big)\Big(\frac{(1-\theta)a^2}{(a-\omega\alpha_1)(a-\omega\alpha_2)-\theta a^2}\Big)^{\lambda t},
	\end{align}}
	and
{\footnotesize	\begin{align*}
		\mathcal{F}\Big\{-(1-\theta)&a^2\Big(1-e^{-\frac{1}{\lambda}\partial_t}\Big)g(\bar{x},t);\alpha_1,\alpha_2\Big\}\\
		&=-(1-\theta)a^2 \Big(\frac{(1-\theta)a^2}{(a-\omega\alpha_1)(a-\omega\alpha_2)-\theta a^2}\Big)^{\lambda t}\Big(1-\Big(\frac{(1-\theta)a^2}{(a-\omega\alpha_1)(a-\omega\alpha_2)-\theta a^2}\Big)^{-1}\Big),
	\end{align*}}
	which coincides with \eqref{q2_lhs}.
	
For $q=3$, we have
{\scriptsize	\begin{align}\label{q3_lhs}
		&\mathcal{F}\Bigl\{\Big(\frac{\partial^3}{\partial x_1\partial x_2\partial x_3}+a\Big(\frac{\partial^2}{\partial x_1\partial x_2}+\frac{\partial^2}{\partial x_1\partial x_3}+\frac{\partial^2}{\partial x_2\partial x_3}\Big)+a^2\Big(\frac{\partial}{\partial x_1}+\frac{\partial}{\partial x_2}+\frac{\partial}{\partial x_3}\Big)\Big)g(\bar{x},t);\alpha_1,\alpha_2,\alpha_3\Bigr\}\nonumber\\
		&=\Big((-\omega\alpha_1)(-\omega\alpha_2)(-\omega\alpha_3)+a((-\omega\alpha_1)(-\omega\alpha_2)+(-\omega\alpha_1)(-\omega\alpha_3)+(-\omega\alpha_2)(-\omega\alpha_3))\nonumber\\
		&\hspace{5cm}+a^2(-\omega\alpha_1-\omega\alpha_2-\omega\alpha_3)\Big)
		\Big(\frac{(1-\theta)a^3}{(a-\omega\alpha_1)(a-\omega\alpha_2)(a-\omega\alpha_3)-\theta a^3}\Big)^{\lambda t},
	\end{align}}
	and
{\scriptsize \begin{align*}
		\mathcal{F}\Big\{-&(1-\theta)a^3\Big(1-e^{-\frac{1}{\lambda}\partial_t}\Big)g(\bar{x},t);\alpha_1,\alpha_2,\alpha_3\Big\}\\
		&=-(1-\theta)a^3 \Big(\frac{(1-\theta)a^3}{(a-\omega\alpha_1)(a-\omega\alpha_2)(a-\omega\alpha_3)-\theta a^3}\Big)^{\lambda t}\Big(1-\Big(\frac{(1-\theta)a^3}{(a-\omega\alpha_1)(a-\omega\alpha_2)(a-\omega\alpha_3)-\theta a^3}\Big)^{-1}\Big),
	\end{align*}}
	which coincides with \eqref{q3_lhs}.

Proceeding inductively, we get the required differential equation in the form of \eqref{de1}.
		
By using Eq. (1.3.29) of Kilbas {\it et al.} (2006) and \eqref{fourier}, we have
	{\small\begin{align*}
		\mathcal{F}\Big\{\frac{\partial^q}{\partial x_1\dots\partial x_q}g(\bar{x},t); \alpha_1,\dots,\alpha_q\Big\}&=\Big(\prod_{j=1}^{q}(-i\alpha_j)\Big)\mathcal{F}\{g(\bar{x},t); \alpha_1,\dots,\alpha_q\}\\
		&=\Big(\prod_{j=1}^{q}(-i\alpha_j)\Big)\Big(\frac{(1-\theta)a^q}{\big(\prod_{j=1}^{q}(a-i\alpha_j)\big)-\theta a^q}\Big)^{\lambda t},\\
	\mathcal{F}\Big\{\frac{\partial^{q-1}}{\partial x_1\dots\partial x_{k-1}\partial x_{k+1}\dots\partial x_q}g(\bar{x},t); \alpha_1,\dots,\alpha_q\Big\}
	&=\Big(\prod_{\substack{j=1\\j\neq k}}^{q}(-i\alpha_j)\Big)\Big(\frac{(1-\theta)a^q}{\big(\prod_{j=1}^{q}(a-i\alpha_j)\big)-\theta a^q}\Big)^{\lambda t},\\
	&\ \ \vdots\\
	\mathcal{F}\Big\{\frac{\partial^2}{\partial x_k\partial x_l}g(\bar{x},t); \alpha_1,\dots,\alpha_q\Big\}&=(-i\alpha_k)(-i\alpha_l)\Big(\frac{(1-\theta)a^q}{\big(\prod_{j=1}^{q}(a-i\alpha_j)\big)-\theta a^q}\Big)^{\lambda t},\\
	\mathcal{F}\Big\{\frac{\partial}{\partial x_k}g(\bar{x},t); \alpha_1,\dots,\alpha_q\Big\}&=(-i\alpha_k)\Big(\frac{(1-\theta)a^q}{\big(\prod_{j=1}^{q}(a-i\alpha_j)\big)-\theta a^q}\Big)^{\lambda t}.
\end{align*}}
Therefore, the Fourier transform of the left hand side of \eqref{de1} is 
{\small\begin{align}\label{LHS}
\mathcal{F}&\Big\{\Big(\sum_{m=1}^{q}a^{q-m}\sum_{1\le j_1<j_2<\dots<j_m\le q}\frac{\partial^m}{\partial x_{j_1}\partial x_{j_2}\dots\partial x_{j_m}}\Big)g(\bar{x},t);\alpha_1,\alpha_2,\dots,\alpha_q\Big\}\nonumber\\	&=\Big(\frac{(1-\theta)a^q}{\big(\prod_{j=1}^{q}(a-i\alpha_j)\big)-\theta a^q}\Big)^{\lambda t}\nonumber\\
&\ \ \cdot
\Big(\prod_{j=1}^{q}(-i\alpha_j)+ a\sum_{k=1}^{q}\prod_{\substack{j=1\\j\neq k}}^{q}(-i\alpha_j)+\dots+a^{q-2}\sum_{\substack{k<l\\1\le k\le q-1}}(-i\alpha_k)(-i\alpha_l)+a^{q-1}\sum_{k=1}^{q}(-i\alpha_k)\Big).
\end{align}}
	Now, the Fourier transform of the right hand side of \eqref{de1} is given by
	\begin{align}\label{RHS}
		&\mathcal{F}\Bigl\{-(1-\theta)a^q\Big(1-e^{-\frac{1}{\lambda}\partial_t}\Big)g(\bar{x},t);\alpha_1,\dots,\alpha_q\Bigr\}\nonumber\\
		&=-(1-\theta)a^q \Big(\Big(\frac{(1-\theta)a^q}{\big(\prod_{j=1}^{q}(a-i\alpha_j)\big)-\theta a^q}\Big)^{\lambda t}-\Big(\frac{(1-\theta)a^q}{\big(\prod_{j=1}^{q}(a-i\alpha_j)\big)-\theta a^q}\Big)^{\lambda t-1}\Big)\nonumber\\
		&=-(1-\theta)a^q \Big(\frac{(1-\theta)a^q}{\big(\prod_{j=1}^{q}(a-i\alpha_j)\big)-\theta a^q}\Big)^{\lambda t}\Big(1-\Big(\frac{\big(\prod_{j=1}^{q}(a-i\alpha_j)\big)-\theta a^q}{(1-\theta)a^q}\Big)\Big).
	\end{align}
Note that on simplification \eqref{LHS} reduces to \eqref{RHS} which establishes the stated result.
\end{proof}
In the next result, we obtain an alternate governing differential equation of the pdf of multivariate gamma subordinator.
The following differential operator will be used:

For parameters $b_1>0,b_2>0$ and $b_3>0$, we consider the following differential operator:
{\footnotesize\begin{equation*}
	\mathcal{D}_{b_1,b_2}^{b_3}\coloneq\sum_{n=1}^{\infty}\frac{(-1)^{n+1}}{n} \sum_{k=0}^{n}{n\choose k}\frac{(-1)^{n-k}}{b_1^k}\sum_{l=0}^{k}{k\choose l}(-b_2)^{k-l}\Big(\prod_{j=1}^{q}\sum_{l_j=0}^{l}{l \choose l_j}b_3^{l_j}\Big)\frac{\partial^{\sum_{j=1}^{q}(l-l_j)}}{\partial x_1^{l-l_1}\partial x_2^{l-l_2}\dots\partial x_q^{l-l_q}}
\end{equation*}}

which acts on any function that is infinitely differentiable.

\begin{theorem}
For $a_1=a_2=\dots=a_q$, the pdf of multivariate gamma subordinator solves the following differential equation:
	\begin{equation}\label{de2}
		\frac{\partial}{\partial t}g(\bar{x}, t)=-\lambda \mathcal{D}_{(1-\theta)a^q,\theta a^q}^a\,g(\bar{x},t),
	\end{equation}
	with initial and boundary conditions
	\begin{equation}\label{ic_bc2}
		\begin{cases}
			g(\bar{x},0)=\delta (\bar{x}),\\
			\lim_{\|\bar{x}\|\to \infty}\,\frac{\partial^{\sum_{j=1}^{q}(l-l_j)}}{\partial x_1^{l-l_1}\partial x_2^{l-l_2}\dots\partial x_q^{l-l_q}}g(\bar{x},t)=0,
		\end{cases}
	\end{equation}
	respectively.
\end{theorem}

\begin{proof}
On taking the Fourier transform of left hand side of \eqref{de2} and by using \eqref{fourier}, we have
{\small	\begin{align*}
		&\mathcal{F}\Big\{\frac{\partial}{\partial t}g(\bar{x},t);\alpha_1,\alpha_2,\dots,\alpha_q\Big\}\\
		&=-\lambda \Big(\frac{\big(\prod_{j=1}^{q}(a-\omega\alpha_j)\big)-\theta a^q}{(1-\theta)a^q}\Big)^{-\lambda t} \ln \Big(\frac{\big(\prod_{j=1}^{q}(a-\omega\alpha_j)\big)-\theta a^q}{(1-\theta)a^q}\Big)\\
		&=-\lambda \Big(\frac{\big(\prod_{j=1}^{q}(a-\omega\alpha_j)\big)-\theta a^q}{(1-\theta)a^q}\Big)^{-\lambda t}\sum_{n=1}^{\infty}\frac{(-1)^{n+1}}{n}\Big(\frac{\big(\prod_{j=1}^{q}(a-\omega\alpha_j)\big)-\theta a^q}{(1-\theta)a^q}-1\Big)^n\\
		&=-\lambda\Big(\frac{\big(\prod_{j=1}^{q}(a-\omega\alpha_j)\big)-\theta a^q}{(1-\theta)a^q}\Big)^{-\lambda t}\sum_{n=1}^{\infty}\frac{(-1)^{n+1}}{n} \sum_{k=0}^{n}{n\choose k}(-1)^{n-k}\Big(\frac{\big(\prod_{j=1}^{q}(a-\omega\alpha_j)\big)-\theta a^q}{(1-\theta)a^q}\Big)^k \\
		&=-\lambda\Big(\frac{\big(\prod_{j=1}^{q}(a-\omega\alpha_j)\big)-\theta a^q}{(1-\theta)a^q}\Big)^{-\lambda t}\sum_{n=1}^{\infty}\frac{(-1)^{n+1}}{n}\\ 
		&\hspace{4.5cm}\cdot\sum_{k=0}^{n}{n\choose k}\frac{(-1)^{n-k}}{((1-\theta)a^q)^k}\sum_{l=0}^{k}{k\choose l}(-\theta a^q)^{k-l} \prod_{j=1}^{q}\sum_{l_j=0}^{l}{l \choose l_j}a^{l_j}(-\omega\alpha_j)^{l-l_j}\\
		&=-\lambda\sum_{n=1}^{\infty}\frac{(-1)^{n+1}}{n} \sum_{k=0}^{n}{n\choose k}\frac{(-1)^{n-k}}{((1-\theta)a^q)^k}\sum_{l=0}^{k}{k\choose l}(-\theta a^q)^{k-l}\Big(\prod_{j=1}^{q}\sum_{l_j=0}^{l}{l \choose l_j}a^{l_j}\Big)\\
		&\hspace{8cm}\cdot \mathcal{F}\Bigl\{\frac{\partial^{\sum_{j=1}^{q}(l-l_j)}}{\partial x_1^{l-l_1}\partial x_2^{l-l_2}\dots\partial x_q^{l-l_q}}g(\bar{x},t);\alpha_1,\alpha_2,\dots,\alpha_q\Bigr\}\\
		&=-\lambda \mathcal{D}_{(1-\theta)a^q,\theta a^q}^a\,\mathcal{F}\{g(\bar{x},t);\alpha_1,\alpha_2,\dots,\alpha_q\}.
	\end{align*}}
It can be established that the pdf $g(\bar{x},t)$ satisfies \eqref{ic_bc2}. This completes the proof.
\end{proof}

\begin{remark}
On substituting $q=2$ in \eqref{de2}, it reduces to the governing differential equation of a bivariate gamma subordinator (see Meoli (2025), Eq. (3.9)).
\end{remark}

\section{Multivariate GCP time-changed by multivariate gamma subordinator}
Kataria and Dhillon (2025) introduced a multivariate version of GCP and studied its various time-changed variants. Here, we study a multivariate GCP with independent components subordinated with  an independent multivariate gamma subordinator. 

Let $\{\bar{M}(t)\}_{t\ge0}$ be a multivariate GCP whose component processes $\{M_i(t)\}_{t\geq0}$'s are independent GCPs which perform independently $k_i$ kinds of jumps with positive rates $\lambda_{ij_i}$, $j_i=1$, $2$, $\dots$, $k_i$. Also, let $\{\bar{G}(t)\}_{t\geq0}$ be a multivariate gamma subordinator whose component processes $\{G_i(t)\}_{t \geq 0}$'s are conditionally independent as discussed in Section 3, and it is independent of  $\{\bar{M}(t)\}_{t\ge0}$. For $i=1,2,\dots,q$, let $\mathscr{M}_i(t)=M_i(G_i(t))$. We define a time-changed process $\{\bar{\mathscr{M}}(t)\}_{t\geq0}$ as follows:
\begin{equation}\label{timechn}
\bar{\mathscr{M}}(t)\coloneqq(\mathscr{M}_1(t), \mathscr{M}_2(t), \dots,\mathscr{M}_q(t)).
\end{equation}
	
From \eqref{multi_gamma}, recall that $G_i(t)=Z_i(t+\lambda^{-1}B^{-}(\lambda t))$, where $\{Z_i(t)\}_{t\ge0}$ is a gamma process with distribution $\Gamma(\lambda t, a_i)$ for all $i=1,2,\dots,q$. To obtain the distributional properties of $\{\bar{\mathscr{M}}(t)\}_{t\geq0}$, we consider the following subordinated L\'evy process: 
\begin{align*}
\bar{M}(\bar{Z}(t))\coloneq(M_1(Z_1(t)),M_2(Z_2(t)),\dots,M_q(Z_q(t))),\ t\ge0,
\end{align*} 
where $\{M_i(t)\}_{t\ge0}$'s are independent of $\{Z_i(t)\}_{t\ge0}$'s.

By using Theorem 6 of Kataria and Dhillon (2025), the state probabilities of $\{\bar{M}(\bar{Z}(t))\}_{t\ge 0}$ can be obtained in the following form:
{\footnotesize	\begin{equation}\label{stateMZ}
h(\bar{n},t)=\mathrm{Pr}\{\bar{M}(\bar{Z}(t))=\bar{n}\}=\prod_{i=1}^{q}\Big(\frac{a_i}{a_i+\lambda_i}\Big)^{\lambda t}\sum_{\Omega(k_i,n_i)}\frac{\Gamma(\eta_i+\lambda t)}{\Gamma(\lambda t)}\prod_{j_i=1}^{k_i}\Big(\frac{\lambda_{ij_i}}{a_i+\lambda_i}\Big)^{n_{ij_i}}\frac{1}{n_{ij_i}!},\ \bar{n}\ge \bar{0},
\end{equation}}
where $\lambda_i=\lambda_{i1}+\lambda_{i2}+\dots+\lambda_{ik_i}$ and $\eta_i=n_{i1}+n_{i2}+\dots+n_{ik_i}$.
	
On applying Proposition 7 of Kataria and Khandakar (2022b) and by using the independence of component processes of $\{\bar{M}(\bar{Z}(t))\}_{t\ge 0}$, its L\'evy measure is given by
\begin{equation}\label{Lev_MZ}
\Pi_{\bar{M}(\bar{Z}(t))}(A_1 \times A_2\times\dots\times A_q)=
\lambda\sum_{i=1}^{q}\sum_{n_i=1}^{\infty}\sum_{\Omega(k_i,n_i)}\Gamma(\eta_i)\mathbb{I}_{\{n_i\in A_i\}}\prod_{j_i=1}^{k_i}\Big(\frac{\lambda_{ij_i}}{a_i+\lambda_i}\Big)^{n_{ij_i}}\frac{1}{n_{ij_i}!}.
\end{equation}
	
Now, let $Y(t)=t+\lambda^{-1}B^- (\lambda t)$, where $B^- (\lambda t)$ is a negative binomial process with parameter $0<\theta<1$. For $u\in\mathbb{R}$, its characteristic function is given by
\begin{align}\label{chFun}
	\mathbb{E}\big(e^{\omega uY(t)}\big)&=\Big(\frac{(1-\theta)e^{\omega u\lambda^{-1}}}{1-\theta e^{\omega u\lambda^{-1}}}\Big)^{\lambda t}\nonumber\\
	&=\exp\Big(t\Big(\omega u+\sum_{k=1}^{\infty}(e^{\omega u\lambda^{-1}k}-1)\frac{\lambda\theta^k}{k}\Big)\Big)\nonumber\\
	&=\exp\Big(t\Big(\omega u+\int_{\mathbb{R}\setminus\{0\}}(e^{\omega ux}-1)\sum_{k=1}^{\infty}\frac{\lambda\theta^k}{k}\delta_{k/\lambda}(\mathrm{d}x)\Big)\Big).
\end{align}
On comparing \eqref{chFun} with \eqref{Lev-Khint}, we conclude that $\{Y(t)\}_{t\ge0}$ is a L\'evy process with unit drift and L\'evy measure
\begin{equation*}\label{Lev_Y}
	\Pi_{Y(t)}(\mathrm{d}x)=\lambda\sum_{k=1}^{\infty}\frac{\theta^k}{k}\delta_{k/\lambda} (\mathrm{d}x).
\end{equation*}

\begin{theorem}
The L\'evy measure of $\{\bar{\mathscr{M}}(t)\}_{t\geq0}$ is given by
{\small	\begin{align}\label{Lev1}
		\Pi_{\bar{\mathscr{M}}(t)}&(A_1 \times A_2\times\dots\times A_q)\nonumber\\
		&=\lambda\sum_{i=1}^{q}\sum_{n_i=1}^{\infty}\sum_{\Omega(k_i,n_i)}\Gamma(\eta_i)\mathbb{I}_{\{n_i\in A_i\}}\prod_{j_i=1}^{k_i}\Big(\frac{\lambda_{ij_i}}{a_i+\lambda_i}\Big)^{n_{ij_i}}\frac{1}{n_{ij_i}!}\nonumber\\
		&\hspace{2cm}+\lambda\sum_{\bar{n}\succ \bar{0}}\sum_{k=1}^{\infty}\frac{\theta^k}{k}\prod_{i=1}^{q}\Big(\frac{a_i}{a_i+\lambda_i}\Big)^{ k}(k)_{\eta_i}
		\sum_{\Omega(k_i,n_i)}\mathbb{I}_{\{n_i\in A_i\}}
		\prod_{j_i=1}^{k_i}\Big(\frac{\lambda_{ij_i}}{a_i+\lambda_i}\Big)^{n_{ij_i}}\frac{1}{n_{ij_i}!},
	\end{align}}
where $\lambda_i=\lambda_{i1}+\lambda_{i2}+\dots+\lambda_{ik_i}$ and $\eta_i=n_{i1}+n_{i2}+\dots+n_{ik_i}$.

\end{theorem}

\begin{proof}
	By using \eqref{stateMZ}, \eqref{Lev_MZ} and Theorem 30.1 of Sato (1999), the L\'evy measure of $\{\bar{\mathscr{M}}(t)\}_{t\geq0}$ can be obtained as follows:
{\footnotesize	\begin{align*}
		\Pi_{\bar{\mathscr{M}}(t)}&(A_1 \times A_2\times\dots\times A_q)\\
		&=\Pi_{\bar{M}(\bar{Z}(t))}(A_1 \times A_2\times\dots\times A_q)
		+\int_{0}^{\infty}\sum_{\bar{n}\succ\bar{0}}h(\bar{n},x)\Big(\prod_{i=1}^{q}\mathbb{I}_{\{n_i\in A_i\}}\Big)\,\Pi_{Y(t)}(\mathrm{d}x)\\
		&=\lambda\sum_{i=1}^{q}\sum_{n_i=1}^{\infty}\sum_{\Omega(k_i,n_i)}\Gamma(\eta_i)\mathbb{I}_{\{n_i\in A_i\}}\prod_{j_i=1}^{k_i}\Big(\frac{\lambda_{ij_i}}{a_i+\lambda_i}\Big)^{n_{ij_i}}\frac{1}{n_{ij_i}!}\\
		&\hspace{2cm}+ \lambda\sum_{\bar{n}\succ \bar{0}}\sum_{k=1}^{\infty}\frac{\theta^k}{k}\prod_{i=1}^{q}
		\sum_{\Omega(k_i,n_i)}\Big(\frac{a_i}{a_i+\lambda_i}\Big)^{ k}\frac{\Gamma(\eta_i+ k)}{\Gamma(k)}
		\Big(\prod_{j_i=1}^{k_i}\Big(\frac{\lambda_{ij_i}}{a_i+\lambda_i}\Big)^{n_{ij_i}}\frac{1}{n_{ij_i}!}\Big)\mathbb{I}_{\{n_i\in A_i\}} ,
	\end{align*}}
which reduces to the required result.	
\end{proof}

\begin{remark}
On substituting $q=1$ in \eqref{Lev1}, we get the L\'evy measure of GCP time-changed with an independent gamma subordinator (see Kataria and Khandakar (2022b)). 
\end{remark}

\begin{proposition}
For $i=1$, $2$, $\dots,q$ and $l=1$, $2$, $\dots,q$, the covariance and the codifference of $\{M_i(G_i(t))\}_{t\geq0}$ and $\{M_l(G_l(t))\}_{t\geq0}$ are given by
{\footnotesize\begin{equation}\label{covvv1}
\operatorname{Cov}(M_i(G_i(t)),M_l(G_l(t)))=\sum_{j_i=1}^{k_i}j_i^2\lambda_{ij_i}\frac{\lambda t}{a_i(1-\theta)}\mathbb{I}_{\{i=j\}}+\sum_{j_i=1}^{k_i}\sum_{j_l=1}^{k_l}j_ij_l\frac{\lambda_{ij_i}\lambda_{lj_l}\lambda  t}{a_ia_l(1-\theta)^2}(\theta \mathbb{I}_{\{i\ne l\}}+\mathbb{I}_{\{i=l\}})
\end{equation}}
and 
{\footnotesize\begin{equation*}
\tau(M_i(G_i(t)),M_l(G_l(t)))=\lambda t\ln \Big(\frac{a_ia_l(1-\theta)}{(a_i-c_i)(a_l-c_l)-\theta a_ia_l}\Big)\mathbb{I}_{\{i\ne l\}}-\lambda t \ln \Big( \frac{a_ia_l(1-\theta)^2}{(a_i(1-\theta)-c_i)(a_l(1-\theta)-c_l)}\Big),
\end{equation*}}
respectively, where $c_i=-\sum_{j_i=1}^{k_i}\lambda_{ij_i}(1-e^{\omega j_i})$ and $c_l=-\sum_{j_l=1}^{k_l}\lambda_{lj_l}(1-e^{-\omega j_l})$.	
\end{proposition}
\begin{proof}
For $i=l$, by using \eqref{cov1} and \eqref{Gi_mean}, we get
\begin{equation}\label{covv1}
\operatorname{Cov}(M_i(G_i(t)),M_l(G_l(t)))=\Big(\sum_{j_i=1}^{k_i}j_i\lambda_{ij_i}\Big)^2\frac{\lambda t}{(a_i(1-\theta))^2}+\sum_{j_i=1}^{k_i}j_i^2\lambda_{ij_i}\frac{\lambda t}{a_i(1-\theta)}.
\end{equation}
For $i\ne l$, we have
{\footnotesize	\begin{align}\label{mean_MG1}
	\mathbb{E}(M_i(G_i(t))& M_l(G_l(t)))\nonumber\\
	&=\sum_{n=0}^{\infty}\mathbb{E}\big(M_i(Z_i(t+\lambda^{-1}B^-(\lambda t))) M_l(Z_l(t+\lambda^{-1}B^-(\lambda t)))\big) (1-\theta)^{\lambda t}{n+\lambda t -1 \choose n}\theta^n\nonumber\\
		&=(1-\theta)^{\lambda t}\sum_{n=0}^{\infty}{n+\lambda t -1 \choose n}\theta^n \mathbb{E}(M_i(1))\mathbb{E}(M_l(1))\mathbb{E}(Z_i(t+\lambda^{-1}n))\mathbb{E}(Z_l(t+\lambda^{-1}n))\nonumber\\
		&=(1-\theta)^{\lambda t}\sum_{n=0}^{\infty}{n+\lambda t -1 \choose n}\theta^n \sum_{j_i=1}^{k_i}j_i\lambda_{ij_i}\sum_{j_l=1}^{k_l}j_l\lambda_{lj_l}\frac{(n+\lambda t)^2}{a_ia_l}\nonumber\\
		&=\sum_{j_i=1}^{k_i}\sum_{j_l=1}^{k_l}j_ij_l\lambda_{ij_i}\lambda_{lj_l}\frac{\lambda t(\lambda t+\theta)}{a_ia_l(1-\theta)^2}.
\end{align}}
Also, from \eqref{Gi_mean}, it follows that
\begin{equation}\label{mean_MG2}
	\mathbb{E}(M_i(G_i(t)))=\sum_{j_i=1}^{k_i}j_i\lambda_{ij_i}\frac{\lambda t}{a_i(1-\theta)}.	
\end{equation}
Thus, for $i\ne l$, from \eqref{mean_MG1} and \eqref{mean_MG2}, we get
\begin{equation}\label{covv2}
	\operatorname{Cov}(M_i(G_i(t)),M_l(G_l(t)))=\sum_{j_i=1}^{k_i}\sum_{j_l=1}^{k_l}j_ij_l\lambda_{ij_i}\lambda_{lj_l}\frac{\lambda t\theta}{a_ia_l(1-\theta)^2}.
\end{equation}
On combining \eqref{covv1} and \eqref{covv2}, we get \eqref{covvv1}.
	
From \eqref{codiff}, we have
\begin{equation}\label{codiff2}
	\tau(M_i(G_i(t)),M_l(G_l(t)))\coloneq\ln\mathbb{E}(e^{\omega(M_i(G_i(t))-M_l(G_l(t)))})-\ln\mathbb{E}(e^{\omega M_i(G_i(t))})-\ln\mathbb{E}(e^{-\omega M_l(G_l(t))}).
\end{equation}

Now, by using \eqref{pgf_gcp} and \eqref{LT_gamma}, we have
\begin{align}\label{co1}
\mathbb{E}\Big(e^{\omega M_i(G_i(t))}\Big)&=\mathbb{E}\Big(\mathbb{E}\Big(e^{\omega M_i(Z_i(t+\lambda^{-1}B^-(\lambda t)))}\big|B^-(\lambda t)\Big)\Big)\nonumber\\
&=\mathbb{E}\Big(\Big(1+\frac{1}{a_i}\sum_{j_i=1}^{k_i}\lambda_{ij_i}(1-e^{\omega j_i})\Big)^{-B^-(\lambda t)-\lambda t}\Big)\nonumber\\
&=\Big(\frac{a_i(1-\theta)}{a_i(1-\theta)+\sum_{j_i=1}^{k_i}\lambda_{ij_i}(1-e^{\omega j_i})}\Big)^{\lambda t}
\end{align}
where the last step follows from \eqref{pgf_negbin}.
Similarly, for $i\ne l$, we get
\begin{align*}
\mathbb{E}\Big(e^{\omega(M_i(G_i(t))-M_l(G_l(t)))}\Big)&=\mathbb{E}\Big(\mathbb{E}\Big(e^{\omega(M_i(G_i(t))-M_l(G_l(t)))}\big|B^-(\lambda t)\Big)\Big)\nonumber\\
&=\mathbb{E}\Big(\Big(\frac{a_ia_l}{(a_i-c_i)(a_l-c_l)}\Big)^{\lambda t+B^-(\lambda t)}\Big)\nonumber\\
&=\Big(\frac{a_ia_l(1-\theta)}{a_ia_l(1-\theta)-(a_ic_l+a_lc_i)+c_ic_l}\Big)^{\lambda t}.\label{codd2}
\end{align*}	
Finally, from \eqref{codiff2}, \eqref{co1} and \eqref{codd2}, we get the required codifference of $\{M_i(G_i(t))\}_{t\geq0}$ and $\{M_l(G_l(t))\}_{t\geq0}$.	
\end{proof}

\begin{proposition}
Let $|u_i|\leq 1$ for all $i=1,2,\dots,q$ and  $\bar{u}=(u_1,u_2,\dots,u_q)$. Then, the pgf of $\{\bar{\mathscr{M}}(t)\}_{t\ge 0}$ is given by
\begin{equation*}\label{pgf_MG}
\mathbf{G}_{\bar{\mathscr{M}}(t)}(\bar{u})=\mathbb{E}\Big(\prod_{i=1}^{q}u_i^{M_i(G_i(t))}\Big)=\Bigg(
\frac{1-\theta}{ \prod_{i=1}^{q}\big(1+\sum_{j_i=1}^{k_i}\frac{\lambda_{ij_i}}{a_i}(1-u_i^{j_i})\big)-\theta}\Bigg)^{\lambda t}.
\end{equation*}
\end{proposition}
\begin{proof}Note that
	\begin{align*}
		\mathbb{E}\Big(\prod_{i=1}^{q}u_i^{M_i(G_i(t))}\Big)&=\mathbb{E}\Big(\mathbb{E}\Big(\prod_{i=1}^{q}u_i^{M_i(G_i(t))}\big|B^-(\lambda t)\Big)\Big)\\
			&=\mathbb{E}\Big(\prod_{i=1}^{q}\mathbb{E}\Big(u_i^{M_i(Z_i(t+\lambda^{-1}B^-(\lambda t)))}\Big|B^-(\lambda t)\Big)\Big)\\
		&=\mathbb{E}\Big(\prod_{i=1}^{q}\mathbb{E}\Big(\exp\Big(-Z_i(t+\lambda^{-1}B^-(\lambda t))\sum_{j_i=1}^{k_i}\lambda_{ij_i}(1-u_i^{j_i})
		\Big)\Big|B^-(\lambda t)\Big)\Big)\\
		&=\mathbb{E}\Big(\prod_{i=1}^{q}\Big(1+\frac{\sum_{j_i=1}^{k_i}\lambda_{ij_i}(1-u_i^{j_i})}{a_i}\Big)^{-\lambda t-B^-(\lambda t)}\Big),
	\end{align*}
	where the penultimate step follows on using \eqref{pgf_gcp}, and in the last step we have used \eqref{LT_gamma}. Finally, the required result follows by using \eqref{pgf_negbin}.
\end{proof}

\begin{remark}
	The pgf  $\mathbf{G}_{\bar{\mathscr{M}}(t)}(\bar{u})$ solves the following differential equation:
	\begin{equation*}
		\frac{\partial}{\partial t}\mathbf{G}_{\bar{\mathscr{M}}(t)}(\bar{u})
		=\lambda \mathbf{G}_{\bar{\mathscr{M}}(t)}(\bar{u})\ln \Bigg(\frac{1-\theta}{ \prod_{i=1}^{q}\big(1+\sum_{j_i=1}^{k_i}\frac{\lambda_{ij_i}}{a_i}(1-u_i^{j_i})\big)-\theta}\Bigg) 
	\end{equation*}
with $\mathbf{G}_{\bar{\mathscr{M}}(0)}(\bar{u})=1$.
\end{remark}

\begin{theorem}\label{state_prob1}
	The state probabilities $p_{\bar{\mathscr{M}}}(\bar{n},t)=\mathrm{Pr}\{\bar{\mathscr{M}}(t)=\bar{n}\}$, $\bar{n}\ge \bar{0}$ are given by
{\footnotesize	\begin{equation}\label{state_pr2}
				p_{\bar{\mathscr{M}}}(\bar{n},t)=(1-\theta)^{\lambda t}\sum_{h=0}^{\infty}\frac{\theta^h}{h!}\big(\lambda t\big)_h \prod_{i=1}^{q}\Big(\frac{a_i}{a_i+\lambda_i}\Big)^{h+\lambda t}\sum_{\Omega(k_i,n_i)}\big(h+\lambda t\big)_{\eta_i}\prod_{j_i=i}^{k_i}\Big(\frac{\lambda_{ij_i}}{a_i+\lambda_i}\Big)^{n_{ij_i}}				
				\frac{1}{n_{ij_i}!},
		\end{equation}}
	where
    $\lambda_i=\lambda_{i1}+\lambda_{i2}+\dots+\lambda_{ik_i}$ and $\eta_i=n_{i1}+n_{i2}+\dots+n_{ik_i}$ for all $i=1,2,\dots,q$. Here, $(x)_k$ denotes the Pochhammer symbol as defined in \eqref{pochha}.
\end{theorem}

\begin{proof}
	For $t\geq0$, by using \eqref{jopmfgcp} and \eqref{gamma_density}, we have
	{\footnotesize
		\begin{align*}
			p_{\bar{\mathscr{M}}}(\bar{n},t)&=\int_{0}^{\infty}\int_{0}^{\infty}\dots\int_{0}^{\infty} \mathrm{Pr}\{M_1(x_1)=n_1,M_2(x_2)=n_2,\dots,M_q(x_q)=n_q\}g(\bar{x},t)\,\mathrm{d}x_1\,\mathrm{d}x_2\,\dots\mathrm{d}x_q\\
			&=\frac{(1-\theta)^{\lambda t}}{\Gamma(\lambda t)}\sum_{h=0}^{\infty}\frac{\theta^h}{h!(\Gamma(h+\lambda t))^{q-1}}\prod_{i=1}^{q}a_i^{h+\lambda t}\sum_{\Omega(k_i,n_i)}\Big(\prod_{j_i=i}^{k_i}\frac{\lambda_{ij_i}^{n_{ij_i}}}{n_{ij_i}!}\Big)\\
			&\hspace{6cm}\cdot\int_{0}^{\infty}x_i^{h+\lambda t+\sum_{j_i=1}^{k_i}n_{ij_i}-1}e^{-(a_i+\sum_{j_i=1}^{k_i}\lambda_{ij_i})x_i}\,\mathrm{d}x_i\\
			&=\frac{(1-\theta)^{\lambda t}}{\Gamma(\lambda t)}\sum_{h=0}^{\infty}\frac{\theta^h}{h!(\Gamma(h+\lambda t))^{q-1}}\prod_{i=1}^{q}a_i^{h+\lambda t}\sum_{\Omega(k_i,n_i)}\Big(\prod_{j_i=i}^{k_i}\frac{\lambda_{ij_i}^{n_{ij_i}}}{n_{ij_i}!}\Big)\\
			&\hspace{7cm}\cdot (a_i+\lambda_i)^{-h-\lambda t-\sum_{j_i=1}^{k_i}n_{ij_i}}\Gamma\Big(h+\lambda t+\sum_{j_i=1}^{k_i}n_{ij_i}\Big)\\
			&=\frac{(1-\theta)^{\lambda t}}{\Gamma(\lambda t)}\sum_{h=0}^{\infty}\frac{\theta^h}{h!(\Gamma(h+\lambda t))^{q-1}}\prod_{i=1}^{q}a_i^{h+\lambda t}\sum_{\Omega(k_i,n_i)}
			\frac{\Gamma(h+\lambda t+\sum_{j_i=1}^{k_i}n_{ij_i})}{(a_i+\lambda_i)^{h+\lambda t}}\prod_{j_i=1}^{k_i}\Big(\frac{\lambda_{ij_i}}{a_i+\lambda_i}\Big)^{n_{ij_i}}\frac{1}{n_{ij_i}!}.		
	\end{align*}}
	This completes the proof.
\end{proof}
\begin{remark}
On substituting $k_1=k_2=1$, $q=2$ and $a_1=a_2=a$ in Theorem \ref{state_prob1}, it reduces to the state probabilities of bivariate Poisson process time-changed by an independent bivariate gamma subordinator (see Meoli (2025), Theorem 3.4).
\end{remark}

Next, we discuss an application of the obtained results to a shock model.
\subsection{An application to a shock model}
Let us consider a system which is exposed to $q$ distinct types of shocks. Its random failure time is denoted by a non-negative, absolutely continuous random variable $T$. The component processes of $\{\bar{\mathscr{M}}(t)\}_{t\geq0}$ which is defined in \eqref{timechn} are used to model the arrival of shocks. Let $S$ denote a random threshold which takes values in $\mathbb{N}$, and is independent of the multivariate process $\{\bar{\mathscr{M}}(t)\}_{t\geq0}$. We assume that the system fails once the total number of shocks reaches $S$. For $i=1,2,\dots,q$, the cause of failure due to shock of type $i$ is denoted by $C=i$, where $C$ is an integer valued random variable. Thus, $T$ is the first hitting time of $\mathscr{M}_1(t)+\mathscr{M}_2(t)+\dots+\mathscr{M}_q(t)$, that is,
	\begin{equation*}
		T= \inf \{t\geq0 : \mathscr{M}_1(t)+\mathscr{M}_2(t)+\dots+\mathscr{M}_q(t)\ge S\}.
	\end{equation*}
	Its pdf is given by 
	\begin{equation*}
		f_T(t)=\sum_{i=1}^{q}f_i(t),\  t\geq0,
	\end{equation*}
	where $f_i(t)$ denotes the sub-density defined as 
	\begin{equation*}
		f_i(t)=\frac{\mathrm{d}}{\mathrm{dt}} \mathrm{Pr}\{T\leq t, C=i\}, \, i\in\{1,2,\dots,q\}.
	\end{equation*}
	Thus, the pmf of $C$ can be obtained as 
	\begin{equation*}
		\mathrm{Pr}\{C=i\}=\int_{0}^{\infty}f_i(t)\,\mathrm{dt}, \  i\in\{1,2,\dots,q\}.
	\end{equation*}
	Also, the pmf and the survival function of random threshold $S$ are given by
	\begin{equation*}
		p_S(k)=P\{S=k\},\ k\in\mathbb{N}
	\end{equation*} 
	and
	\begin{equation*}
		\bar{F}_S(k)=P\{S>k\},\, k\in\mathbb{N}_0,
	\end{equation*} 
	respectively. Note that the survival function of random failure time $T$ is defined as $\bar{F}_T(t)=\mathrm{Pr}\{T> t\}$, $t\ge 0$. It is given by 
	\begin{equation}\label{surv}
		\bar{F}_T(t)=\sum_{k=0}^{\infty}\bar{F}_S(k)\sum_{n_1+n_2+\dots+n_q=k}p_{\bar{\mathscr{M}}}(\bar{n},t), \, t\geq0,  
	\end{equation}
with $\bar{F}_S(0)=1.$

For $\bar{n}\ge \bar{0}$, the intensity of occurrence of shock of type $r$ due to a jump of size $l_r$ is given by $R_{rl_r}$. It is known as the hazard rate which is defined as
{\footnotesize	\begin{equation}\label{hazard}
R_{rl_r}(\bar{n};t)=\lim_{h\to 0+}\frac{1}{h}\mathrm{Pr}\Big(\mathscr{M}_r(t+h)=n_r +l_r,\cap_{\substack{k=1\\ k\neq r}}^q\{\mathscr{M}_k(t+h)=n_k\}\big|\cap_{k=1}^{q}\mathscr{M}_k(t)=n_k \Big),
\end{equation}}
where $r=1,2,\dots,q$ and $l_r=1,2,\dots,k_r$.
		
\begin{theorem}
For $r=1,2,\dots,q$ and $l_r=1,2,\dots,k_r$, the hazard rate $R_{rl_r}(\bar{n};t)$, $\bar{n}\ge \bar{0}$ is given by
	\begin{equation}\label{hazard2}
		R_{rl_r}(\bar{n};t)=\frac{A(\theta,\bar{a},t)}{p_{\bar{\mathscr{M}}}(\bar{n},t)},\ t\geq 0.
	\end{equation}
	Here, $p_{\bar{\mathscr{M}}}(\bar{n},t)$ is the pmf of multivariate process $\{\bar{\mathscr{M}}(t)\}_{t\ge 0}$ and {\scriptsize \begin{align*}
		A(\theta,\bar{a},t)&=(1-\theta)^{\lambda t}\sum_{h_1=0}^{\infty}\sum_{h_2=0}^{\infty}
		\frac{\lambda \theta^{h_1+h_2}}{h_1!h_2!}\sum_{\Omega(k_r,l_r)}\Big(\prod_{j_r=1}^{k_r}\frac{\lambda_{rj_r}^{l_{rj_r}}}{l_{rj_r}!}\Big)
		\Big(\prod_{i=1}^{q}a_i^{h_1+h_2+\lambda t}\sum_{\Omega(k_i,n_i)}\prod_{j_i=1}^{k_i}\frac{\lambda_{ij_i}^{n_{ij_i}}}{n_{ij_i}!}\Big)\nonumber\\ &\cdot(a_r+\lambda_r)^{-h_1-h_2-\lambda t-\eta_r-\xi_r}(h_1+\xi_r-1)!\big(\lambda t\big)_{h_2+\eta_r}
		\prod_{\substack{i=1\\i\neq r}}^{q}\sum_{\Omega(k_i,n_i)}(a_i+\lambda_i)^{-h_1-h_2-\lambda t-\eta_i}\big(h_2+\lambda t\big)_{\eta_i},
	\end{align*}}
	where $\xi_r=\sum_{j_r=1}^{k_r}l_{rj_r}$, $\lambda_i=\sum_{j_i=1}^{k_i}\lambda_{ij_i}$ and $\eta_i=\sum_{j_i=1}^{k_i}n_{ij_i}$.
\end{theorem}

\begin{proof}
	For $r=1,2,\dots,q$, we have the following equivalent form of \eqref{hazard}:
	{\footnotesize	\begin{equation}\label{lim}
			R_{rl_r}(\bar{n};t)=\lim_{\tau\to t}\frac{\mathrm{Pr}\Big(\mathscr{M}_r(\tau)=n_r +l_r,\cap_{\substack{k=1\\ k\neq r}}^q\{\mathscr{M}_k(\tau)=n_k\},\cap_{k=1}^{q}\{\mathscr{M}_k(t)=n_k \} \Big)}{(\tau-t)p_{\bar{\mathscr{M}}}(\bar{n},t)},
	\end{equation}}
	where $\bar{n}\in \mathbb{N}_0^q$.
	
	By using the independent and stationary increments property of the component processes of $\{\bar{G}(t)\}_{t\geq0}$ and $\{\bar{M}(t)\}_{t\geq0}$, we have
{\scriptsize	\begin{align}
		&\mathrm{Pr}\Big(\mathscr{M}_r(\tau)=n_r +l_r,\cap_{\substack{k=1\\ k\neq r}}^q\{\mathscr{M}_k(\tau)=n_k\},\cap_{k=1}^{q}\{\mathscr{M}_k(t)=n_k \} \Big)\nonumber\\
		&=\int_{x_{11}=0}^{\infty}\int_{x_{12}=0}^{x_{11}}\dots \int_{x_{q1}=0}^{\infty}\int_{x_{q2}=0}^{x_{q1}}\mathrm{Pr}\Big(M_r(x_{r1})=n_r+l_r,\cap_{\substack{k=1\\k\neq r}}^q\{M_k(x_{k1})=n_k\},\cap_{k=1}^{q}\{M_k(x_{k2})=n_k\}\Big)\nonumber\\
		&\hspace{8.5cm}\cdot\mathrm{Pr}\Big(\cap_{k=1}^{q}\{G_k(\tau)\in\mathrm{d}x_{k1}\},\cap_{k=1}^{q}\{G_k(t)\in\mathrm{d}x_{k2}\}\big)\nonumber\\
		&=\int_{x_{11}=0}^{\infty}\int_{x_{12}=0}^{x_{11}}\dots \int_{x_{q1}=0}^{\infty}\int_{x_{q2}=0}^{x_{q1}}\mathrm{Pr}\Big(M_r(x_{r1})=n_r+l_r,\cap_{\substack{k=1\\k\neq r}}^q\{M_k(x_{k1})=n_k\},\cap_{k=1}^{q}\{M_k(x_{k2})=n_k\}\Big)\nonumber\\
		&\hspace{2.5cm}\cdot g((x_{11}-x_{12},x_{21}-x_{22},\dots,x_{q1}-x_{q2}), \tau-t)\, g((x_{12},x_{22},\dots,x_{q2}), t)\,\mathrm{d}x_{q2}\,\mathrm{d}x_{q1}\dots\mathrm{d}x_{12}\,\mathrm{d}x_{11}\nonumber\\
		&=\int_{x_{11}=0}^{\infty}\int_{x_{12}=0}^{x_{11}}\dots \int_{x_{q1}=0}^{\infty}\int_{x_{q2}=0}^{x_{q1}}\mathrm{Pr}\Big(M_r(x_{r1}-x_{r2})=l_r,\cap_{\substack{k=1\\ k\neq r}}^q\{M_k(x_{k1}-x_{k2})=0\}\Big)\mathrm{Pr}\big\{\cap_{k=1}^{q}M_k(x_{k2})=n_k\big\}\nonumber\\
		&\hspace{2.5cm}\cdot g((x_{11}-x_{12},x_{21}-x_{22},\dots,x_{q1}-x_{q2}), \tau-t)\, g((x_{12},x_{22},\dots,x_{q2}), t)\,\mathrm{d}x_{q2}\,\mathrm{d}x_{q1}\dots\mathrm{d}x_{12}\,\mathrm{d}x_{11},\nonumber
	\end{align}}
	where we have used the independence of the component processes of $\{\bar{M}(t)\}_{t\geq0}$.
	
	By using \eqref{jopmfgcp} and \eqref{gamma_density}, we have
{\footnotesize	\begin{align*}
		&\mathrm{Pr}\Big(\mathscr{M}_r(\tau)=n_r +l_r,\cap_{\substack{k=1\\ k\neq r}}^q\{\mathscr{M}_k(\tau)=n_k\},\cap_{k=1}^{q}\{\mathscr{M}_k(t)=n_k \} \Big)\nonumber\\
		&=\int_{x_{11}=0}^{\infty}\int_{x_{12}=0}^{x_{11}}\dots \int_{x_{q1}=0}^{\infty}\int_{x_{q2}=0}^{x_{q1}}
		\Big(\sum_{\Omega(k_r,l_r)}\prod_{j_r=1}^{k_r}\frac{(\lambda_{rj_r}(x_{r1}-x_{r2}))^{l_{rj_r}}}{l_{rj_r}!}\Big)\nonumber\\
		&\cdot\Big(\prod_{i=1}^{q}\sum_{\Omega(k_i,n_i)}\prod_{j_i=1}^{k_i}\frac{(\lambda_{ij_i}x_{i2})^{n_{ij_i}}}{n_{ij_i}!}e^{-\lambda_{ij_i}x_{i1}}\Big)
		\frac{(1-\theta)^{\lambda \tau}}{\Gamma(\lambda (\tau-t))\Gamma(\lambda t)}\sum_{h_1=0}^{\infty}\sum_{h_2=0}^{\infty}
		\frac{\theta^{h_1+h_2}}{h_1!h_2!(\Gamma(h_1+\lambda (\tau-t))\Gamma(h_2+\lambda t))^{q-1}}\nonumber\\		
		&\hspace{4.2cm}\cdot\Big(\prod_{i=1}^{q}a_i^{h_1+h_2+\lambda\tau}(x_{i1}-x_{i2})^{h_1+\lambda (\tau-t)-1}x_{i2}^{h_2+\lambda t-1}e^{-a_ix_{i1}}\Big)\,\mathrm{d}x_{q2}\,\mathrm{d}x_{q1}\dots\,\mathrm{d}x_{12}\,\mathrm{d}x_{11}\nonumber\\
		&=\frac{(1-\theta)^{\lambda \tau}}{\Gamma(\lambda (\tau-t))\Gamma(\lambda t)}\sum_{h_1=0}^{\infty}\sum_{h_2=0}^{\infty}
		\frac{\theta^{h_1+h_2}}{h_1!h_2!(\Gamma(h_1+\lambda (\tau-t))\Gamma(h_2+\lambda t))^{q-1}}\nonumber\\
		&\hspace{1cm}\cdot\sum_{\Omega(k_r,l_r)}\Big(\prod_{j_r=1}^{k_r}\frac{\lambda_{rj_r}^{l_{rj_r}}}{l_{rj_r}!}\Big)
		\Big(\prod_{i=1}^{q}a_i^{h_1+h_2+\lambda\tau}\sum_{\Omega(k_i,n_i)}\prod_{j_i=1}^{k_i}\frac{\lambda_{ij_i}^{n_{ij_i}}}{n_{ij_i}!}\Big)\int_{x_{11}=0}^{\infty}\int_{x_{12}=0}^{x_{11}}\dots \int_{x_{q1}=0}^{\infty}\int_{x_{q2}=0}^{x_{q1}}
		(x_{r1}-x_{r2})^{\xi_r}\nonumber\\
		&\hspace{3.7cm}\cdot\Big(\prod_{i=1}^{q}x_{i2}^{\eta_i+h_2+\lambda t-1}(x_{i1}-x_{i2})^{h_1+\lambda(\tau-t)-1}e^{-(a_i+\lambda_i)x_{i1}}\Big)\,\mathrm{d}x_{q2}\,\mathrm{d}x_{q1}\dots\,\mathrm{d}x_{12}\,\mathrm{d}x_{11}\nonumber\\
		&=\frac{(1-\theta)^{\lambda \tau}}{\Gamma(\lambda (\tau-t))\Gamma(\lambda t)}\sum_{h_1=0}^{\infty}\sum_{h_2=0}^{\infty}
		\frac{\theta^{h_1+h_2}}{h_1!h_2!(\Gamma(h_1+\lambda (\tau-t))\Gamma(h_2+\lambda t))^{q-1}}
		\sum_{\Omega(k_r,l_r)}\Big(\prod_{j_r=1}^{k_r}\frac{\lambda_{rj_r}^{l_{rj_r}}}{l_{rj_r}!}\Big)\nonumber\\
		&\hspace{.8cm}\cdot\Big(\prod_{i=1}^{q}a_i^{h_1+h_2+\lambda\tau}\sum_{\Omega(k_i,n_i)}\prod_{j_i=1}^{k_i}\frac{\lambda_{ij_i}^{n_{ij_i}}}{n_{ij_i}!}\Big)
		(a_r+\lambda_r)^{-h_1-h_2-\lambda\tau-\eta_r-\xi_r}\Gamma(h_1+\lambda(\tau-t)+\xi_r)\Gamma(h_2+\lambda t+\eta_r)\nonumber\\
		&\hspace{6.4cm}\cdot\prod_{\substack{i=1\\i\neq r}}^{q}(a_i+\lambda_i)^{-h_1-h_2-\lambda\tau-\eta_i}\Gamma(h_1+\lambda(\tau-t))\Gamma(h_2+\lambda t+\eta_i)\nonumber\\
		&=(1-\theta)^{\lambda \tau}\sum_{h_1=0}^{\infty}\sum_{h_2=0}^{\infty}
		\frac{\theta^{h_1+h_2}}{h_1!h_2!}\sum_{\Omega(k_r,l_r)}\Big(\prod_{j_r=1}^{k_r}\frac{\lambda_{rj_r}^{l_{rj_r}}}{l_{rj_r}!}\Big)
		\Big(\prod_{i=1}^{q}a_i^{h_1+h_2+\lambda\tau}\sum_{\Omega(k_i,n_i)}\prod_{j_i=1}^{k_i}\frac{\lambda_{ij_i}^{n_{ij_i}}}{n_{ij_i}!}\Big)\nonumber\\ &\cdot(a_r+\lambda_r)^{-h_1-h_2-\lambda\tau-\eta_r-\xi_r}\big(\lambda(\tau-t)\big)_{h_1+\xi_r}\big(\lambda t\big)_{h_2+\eta_r}
		\prod_{\substack{i=1\\i\neq r}}^{q}\sum_{\Omega(k_i,n_i)}(a_i+\lambda_i)^{-h_1-h_2-\lambda\tau-\eta_i}\big(h_2+\lambda t\big)_{\eta_i},\label{num}
	\end{align*}}
	where the penultimate step follows on using formula 3.191.1 of Gradshteyn and Ryzhik (2014). Finally, on substituting \eqref{state_pr2} and \eqref{num} in \eqref{lim}, we get the required result.
\end{proof}

\begin{remark}
For $q=2$, $k_1=k_2=1$ and $a_1=a_2=a$, \eqref{hazard2} reduces to the hazard rates obtained in Theorem 3.5 of Meoli (2025).
\end{remark}

	\begin{remark}
		For $r=1,2,\dots,q$, $l_r=1,2,\dots,k_r$ and $t\ge 0$, the failure sub-densities are given by (see Kataria and Dhillon (2025), Eq. (5.4))
\begin{equation}\label{sub-den}
f_r(t)=\sum_{k=1}^{\infty}p_S(k)\sum_{k-k_r\leq\sum_{i=1}^{q}n_i\leq k-1}\mathrm{Pr}\Big(\cap_{i=1}^{q}\{\mathscr{M}_i(t)=n_i\}\big)\sum_{l_r=k-\sum_{i=1}^{q}n_i}^{k_r}R_{rl_r}(\bar{n};t).
\end{equation}
		
Now, on substituting \eqref{hazard2} in \eqref{sub-den}, we get
		{\footnotesize \begin{align*}
			f_r(t)&=(1-\theta)^{\lambda t}\sum_{k=1}^{\infty}p_S(k)\sum_{k-k_r\leq\sum_{i=1}^{q}n_i\leq k-1}\,\,\sum_{l_r=k-\sum_{i=1}^{q}n_i}^{k_r}
				\sum_{h_1=0}^{\infty}\sum_{h_2=0}^{\infty}
				\frac{\lambda \theta^{h_1+h_2}}{h_1!h_2!}\\
				&\hspace{1cm} \cdot\sum_{\Omega(k_r,l_r)}\Big(\prod_{j_r=1}^{k_r}\frac{\lambda_{rj_r}^{l_{rj_r}}}{l_{rj_r}!}\Big)
				\Big(\prod_{i=1}^{q}a_i^{h_1+h_2+\lambda t}\sum_{\Omega(k_i,n_i)}\prod_{j_i=1}^{k_i}\frac{\lambda_{ij_i}^{n_{ij_i}}}{n_{ij_i}!}\Big)(a_r+\lambda_r)^{-h_1-h_2-\lambda t-\eta_r-\xi_r}\nonumber\\ 
				&\hspace{1cm} \cdot(h_1+\xi_r-1)!\big(\lambda t\big)_{h_2+\eta_r}
				\prod_{\substack{i=1\\i\neq r}}^{q}\sum_{\Omega(k_i,n_i)}(a_i+\lambda_i)^{-h_1-h_2-\lambda t-\eta_i}\big(h_2+\lambda t\big)_{\eta_i}.
			\end{align*}}
	\end{remark}
	
	\begin{remark}
The survival function of failure time $T$ is given by
{\scriptsize		\begin{equation}\label{surv2}
\bar{F}_T(t)=(1-\theta)^{\lambda t}\sum_{k=0}^{\infty}\bar{F}_S(k)\sum_{n_1+n_2+\dots+n_q=k}\sum_{h=0}^{\infty}\frac{\theta^h}{h!}\big(\lambda t\big)_h \prod_{i=1}^{q}\Big(\frac{a_i}{a_i+\lambda_i}\Big)^{h+\lambda t}\sum_{\Omega(k_i,n_i)}\big(h+\lambda t\big)_{\eta_i}\prod_{j_i=i}^{k_i}\Big(\frac{\lambda_{ij_i}}{a_i+\lambda_i}\Big)^{n_{ij_i}}				
\frac{1}{n_{ij_i}!},
		\end{equation}}
where we have used \eqref{state_pr2} and \eqref{surv}.
\end{remark}

\subsubsection{Some particular cases for the survival function of random failure time}
Here, we present two representative examples of the survival function of the random failure time
 $T$ derived under specific assumptions on the survival distribution $\bar{F}_S(k)=\mathrm{Pr}\{S>k\}$, $k\ge 0$ of the random threshold $S$. Also, we plot the survival function for different choices of parameters corresponding to the following cases:

\noindent I.
 Suppose the random threshold $S$ follows Geometric$(p)$ distribution, that is, $\bar{F}_S(k)=(1-p)^k $, $0<p\le 1$.

On substituting $q=2, k_1=1$ and $k_2=2$ in \eqref{surv2}, we get
{\footnotesize \begin{align}\label{Geo}
\bar{F}_T(t)&=(1-\theta)^{\lambda t}\sum_{k=0}^{\infty}(1-p)^k\sum_{n_1+n_{21}+2n_{22}=k}	\sum_{h=0}^{\infty}\frac{\theta^h}{h!}\big(\lambda t\big)_h \big(h+\lambda t\big)_{n_1}\big(h+\lambda t\big)_{n_{21}+n_{22}}\nonumber\\ 
&\cdot\Big(\frac{a_1a_2}{(a_1+\lambda_1)(a_2+\lambda_{21}+\lambda_{22})}\Big)^{h+\lambda t}\frac{1}{n_1!n_{21}!n_{22}!} \frac{\lambda_1^{n_1}\lambda_{21}^{n_{21}} \lambda_{22}^{n_{22}}}{(a_1+\lambda_1)^{n_1}(a_2+\lambda_{21}+\lambda_{22})^{n_{21}+n_{22}}}.
\end{align}}

\noindent II. Let $S$ has Hypergeometric$(N,K,n)$ distribution, that is,
{\small \begin{align}\label{hyp}
		\bar{F}_S(k)=\frac{{n\choose k+1}{N-n \choose K-k-1}}{{N\choose K}}\,_3F_2(1, k+1-K, k+1-n; k+2, N+k+2-K-n; 1),
\end{align}}
where $N\in\mathbb{N}_0$, $K\in\{0,1,2,\dots,N\}$ and $n\in\{0,1,2,\dots,N\}$.\\
Now, on substituting $q=2, k_1=1$ and $k_2=2$ in \eqref{surv2}, and using \eqref{hyp}, we get
{\footnotesize \begin{align}\label{hypergeo}
		\bar{F}_T(t)&=(1-\theta)^{\lambda t}\sum_{k=0}^{\infty}\frac{{n\choose k+1}{N-n \choose K-k-1}}{{N\choose K}}\,_3F_2(1, k+1-K, k+1-n; k+2, N+k+2-K-n; 1)\nonumber\\
		&\ \ \cdot\sum_{n_1+n_{21}+2n_{22}=k}	\sum_{h=0}^{\infty}\frac{\theta^h}{h!}\big(\lambda t\big)_h \big(h+\lambda t\big)_{n_1}\big(h+\lambda t\big)_{n_{21}+n_{22}}\Big(\frac{a_1a_2}{(a_1+\lambda_1)(a_2+\lambda_{21}+\lambda_{22})}\Big)^{h+\lambda t}\nonumber\\ 
		&\hspace{6cm}\cdot\frac{1}{n_1!n_{21}!n_{22}!} \frac{\lambda_1^{n_1}\lambda_{21}^{n_{21}} \lambda_{22}^{n_{22}}}{(a_1+\lambda_1)^{n_1}(a_2+\lambda_{21}+\lambda_{22})^{n_{21}+n_{22}}}.
\end{align}}

\begin{figure}[H]
	\begin{center}
		\includegraphics[width=15cm]{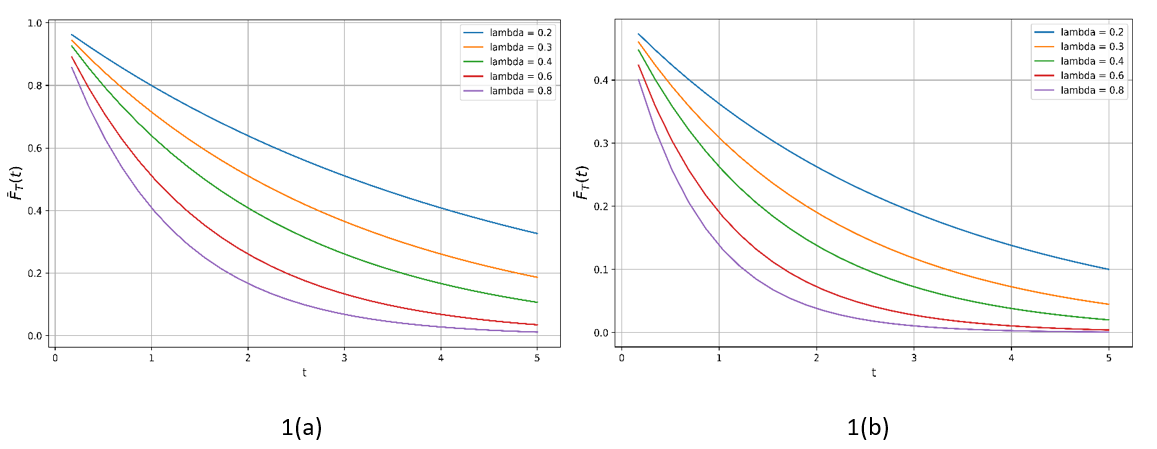}
		\vspace{-.3cm}\caption{Plots 1(a), 1(b) represent the survival function \eqref{Geo} and \eqref{hypergeo} for the parameter values $\theta=0.5$, $a_1=a_2=1$, $\lambda_1=0.5$, $\lambda_{21}=0.5$, $\lambda_{22}=0.5$ and: $p=0.5$ for 1(a) and $N=2$, $K=1$, $n=1$ for 1(b) respectively.}\label{slide1}
	\end{center}

\end{figure}

\begin{figure}
	\includegraphics[width=15cm]{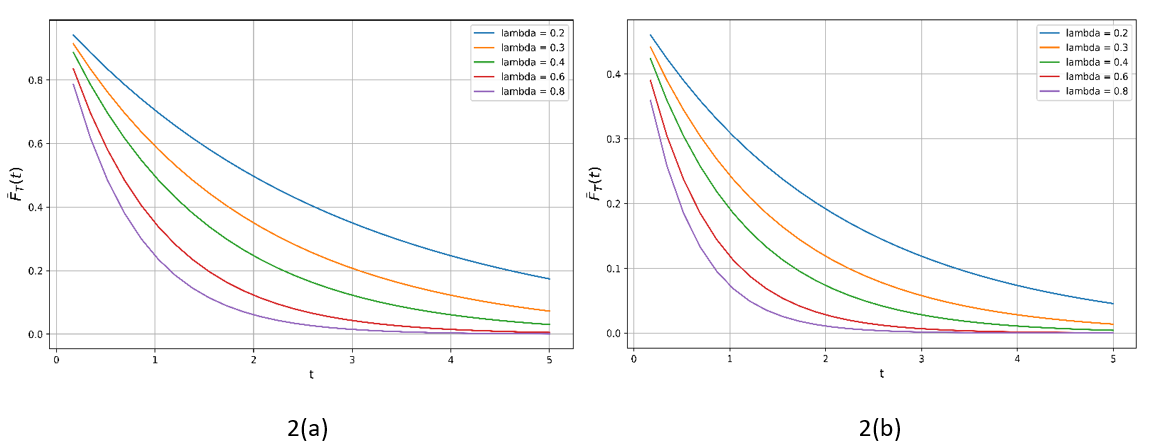}
	\vspace{-.3cm}\caption{Plots 2(a), 2(b) represent the survival function \eqref{Geo} and \eqref{hypergeo} for the parameter values $\theta=0.5$, $a_1=a_2=1$, $\lambda_1=1$, $\lambda_{21}=1$, $\lambda_{22}=1$ and: $p=0.5$ for 2(a) and $N=2$, $K=1$, $n=1$ for 2(b) respectively.}\label{slide2}
	
\end{figure}

\begin{figure}
	\includegraphics[width=15cm]{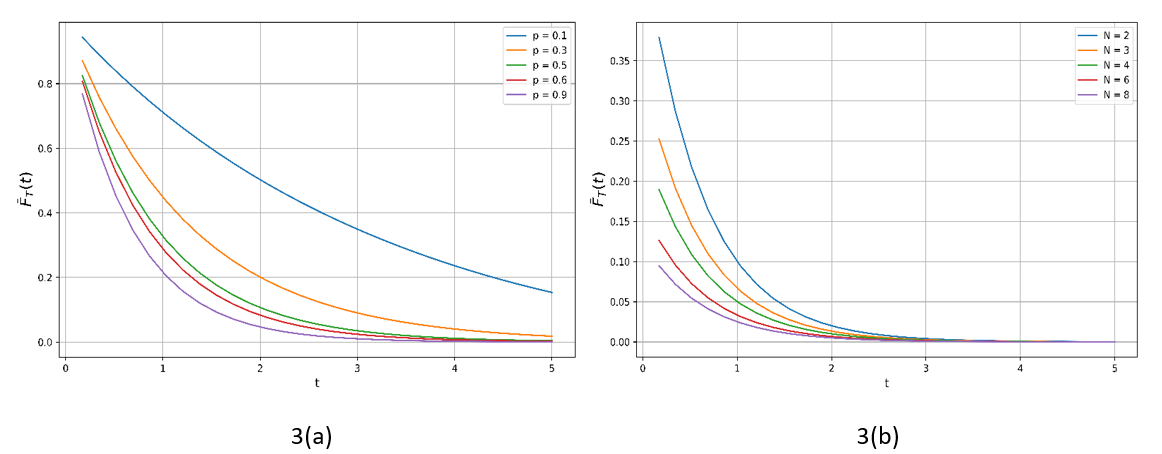}
	\vspace{-.3cm}\caption{Plots 3(a), 3(b) represents the survival function \eqref{Geo} and \eqref{hypergeo} for the parameter values $\theta=0.5$, $a_1=a_2=1$, $\lambda=1$, $\lambda_1=0.5$, $\lambda_{21}=0.5$, $\lambda_{22}=0.5$ and $K=1$, $n=1$ (for 3(b)) respectively.}\label{slide3}
	
\end{figure}

\begin{figure}[H]
	\includegraphics[width=15cm]{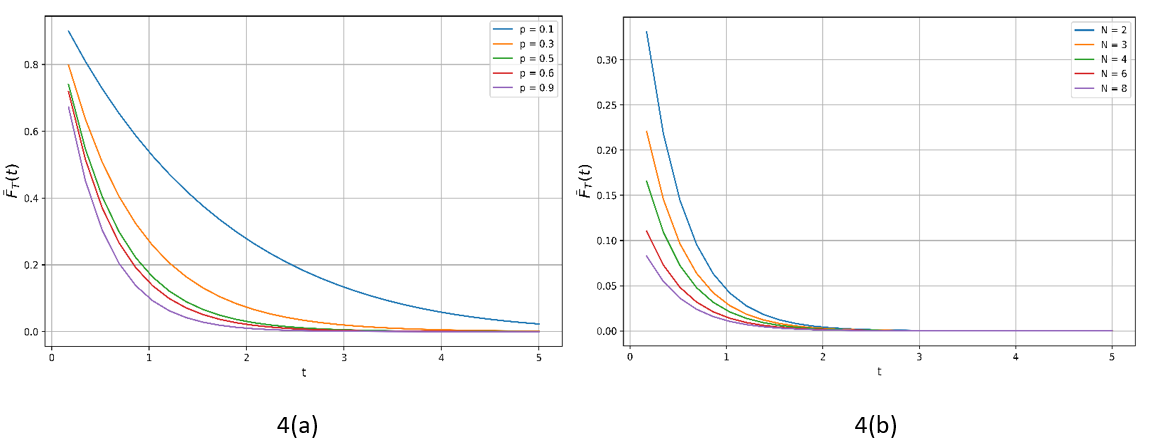}
	\vspace{-.3cm}\caption{Plots 4(a), 4(b) represents the survival function \eqref{Geo} and \eqref{hypergeo} for the parameter values $\theta=0.5$, $a_1=a_2=1$, $\lambda=1$, $\lambda_1=1$, $\lambda_{21}=1$, $\lambda_{22}=1$ and $K=1$, $n=1$ (for 4(b)) respectively.}\label{slide4}
	
\end{figure}

\begin{remark} 
	From the plots of the survival function corresponding to the two cases discussed above, we have the following observations:
	
	\noindent(i) From Figure \ref{slide1} and Figure \ref{slide2}, it is observed that the survival function $\bar{F}_T(\cdot)$ decreases with the increasing values of $\lambda$ for Case I and Case II.\\
	\noindent(ii) From Figure \ref{slide3} and Figure \ref{slide4}, it is observed that the survival function $\bar{F}_T(\cdot)$ decreases with the increasing values of $p$ in 3(a) and 4(a). Also, it decreases with with the increasing values of $N$ in 3(b) and 4(b).\\
\end{remark}

\end{document}